\documentclass[jsl]{asl}

\usepackage[usenames,dvipsnames]{color}

 \usepackage[colorlinks,citecolor=blue,urlcolor=blue, linkcolor=blue]{hyperref}

\usepackage{color,soul}
 \usepackage{xspace}
 \usepackage{graphicx}

 \usepackage[all]{xy}

\usepackage{tikz}
\usetikzlibrary{arrows,snakes,positioning,backgrounds,shadows}

\newtheorem{theorem}{Theorem}[section]

\newtheorem{thm}[theorem]{Theorem}

\newtheorem{fact}[theorem]{Fact}

\newtheorem{prop}[theorem]{Proposition}

\newtheorem{cor}[theorem]{Corollary}

\newtheorem{question}[theorem]{Question}

\theoremstyle{definition}
\newtheorem{definition}[theorem]{Definition}
\newtheorem{claim}[theorem]{Claim}
\newtheorem{lemma}[theorem]{Lemma}

\newtheorem{remark}[theorem]{Remark}

 \usepackage{cleveref}

\crefname{claim}{claim}{claims}
\Crefname{claim}{Claim}{Claims}

\crefname{subsection}{subsection}{subsections}
\Crefname{subsection}{Subsection}{Subsections}

\crefname{question}{question}{questions}
\Crefname{question}{Question}{Questions}

\newcommand{\NN}{\omega}

\newcommand{\sub}{\subseteq}
\newcommand{\sN}[1]{_{#1\in \NN}}
\newcommand{\uhr}[1]{\! \upharpoonright_{#1}}

\renewcommand{\P}{\mathcal P}

\newcommand{\bi}{\begin{itemize}}
\newcommand{\ei}{\end{itemize}}
\newcommand{\bc}{\begin{center}}
\newcommand{\ec}{\end{center}}

\newcommand{\ES}{\emptyset}

\newcommand{\tp}[1]{2^{#1}}
\newcommand{\ex}{\exists}
\newcommand{\fa}{\forall}

\newcommand{\la}{\langle}
\newcommand{\ra}{\rangle}

\newcommand{\cantor}{{}^\omega 2}

\newcommand{\leT}{\le_{\mathrm{T}}}

\renewcommand{\P}{\mathcal P}

\newcommand{\n}{\noindent}

\newcommand{\vsps}{\vspace{3pt}}

\newcommand{\sss}{\sigma}
\newcommand{\aaa}{\alpha}

\renewcommand{\hat}{\widehat}
\newcommand{\cat}{\widehat{\phantom{\alpha}}}

\newcommand \seq[1]{{\left\langle{#1}\right\rangle}}

\newcommand\+[1]{\mathcal{#1}}

\newcommand{\wt}{\widetilde}
\newcommand{\ol}{\overline}
\newcommand{\uul}{\underline}

\newcommand{\lra}{\leftrightarrow}
\newcommand{\LR}{\Leftrightarrow}
\newcommand{\RA}{\Rightarrow}

\newcommand{\s}{\sigma}
\newcommand{\rest}[1]{\! \upharpoonright_{#1}}

\newcommand{\sssl}{\ensuremath{|\sigma|}}

\newcommand{\frb}{\mathfrak{b}}
\newcommand{\frd}{\mathfrak{d}}

\newif\ifbenoit
\benoittrue 

\definecolor{lightred}{rgb}{1,.60,.60}

\newcommand{\NumberQED}[1]{
\renewcommand\qedsymbol{\ensuremath{\openbox}$_{\ref{#1}}$}	
}

\title{Muchnik degrees   and cardinal characteristics}
\author{Benoit Monin}
\address{LACL, IUT S\'enart-Fontainebleau, Universit\'e Paris-Est Cr\'eteil, France}
\email{benoit.monin@computability.fr}

\author{Andr\'e Nies}
\address{University of Auckland, New Zealand}
\email{andre@cs.auckland.ac.nz}

\begin{document}
 

\begin{abstract} 
A mass problem is a set of functions   $\omega\to \omega$.   For mass problems $\+ C, \+ D$, one  says that $\+ C$ is Muchnik reducible to $\+ D$ if each function in $\+ C$ is computed by  a function in $\+ D$.  In this paper we study some highness properties of Turing oracles, which we view as mass problems. We  compare them  with respect to Muchnik reducibility and its uniform strengthening, Medvedev reducibility.

 For $p \in [0,1]$  let $\+ D(p)$ be the mass problem of infinite bit sequences~$y$ (i.e., $\{0,1\}$-valued functions)  such that for each computable bit sequence~$x$,   the   bit sequence $ x \lra y$  has  asymptotic  lower  density   at most~$p$ (where $x \lra y$  has a $1$ in position $n$ iff  $x(n) =  y(n)$).  
 We show that all members of  this  family of mass problems parameterized by a real $p$ with  $0 <  p<1/2 $ have the same complexity in the sense of Muchnik reducibility.     We prove this by showing    Muchnik equivalence of the problems  $\+ D(p)$ with  the mass problem $\mathrm{IOE}(\tp { \tp n})$; here for an order function $h$, the mass problem $\mathrm{IOE}(h)$ consists  of the functions $f$ that agree infinitely often with each computable function bounded by~$h$.    This result also yields a new version of the proof to  of the  affirmative answer to the 	``Gamma question'' due to the first author:    $\Gamma(A)< 1/2$ implies $\Gamma(A)=0$ for each Turing oracle~$A$.  
 
 As a dual of  the problem  $\+ D(p)$, define   $\+ B(p)$, for   $0 \le p < 1/2$,  to be  the set of bit sequences $y$  such that   $\uul \rho (x \lra y) > p$ for each computable  set~$x$. We prove  that the  Medvedev (and hence Muchnik) complexity  of the  mass problems $\+ B(p)$ is the same for all $p \in (0, 1/2)$, by showing that they are Medvedev equivalent to  the mass problem of functions bounded by $\tp{\tp n}$ that are almost everywhere different from each computable function. 
 
 Next,  together with Joseph Miller,   we obtain a proper hierarchy of the mass problems of type  $\mathrm{IOE}$: we show  that for any order function~$g$ there exists a faster growing  order  function $h $ such that $\mathrm{IOE}(h) $ is strictly  above  $\mathrm{IOE}(g)$ in the sense of Muchnik reducibility.

 We study cardinal characteristics in the sense of set theory that are analogous to the highness properties above. For instance, $\frd (p)$  is the least size of a set $G$ of bit sequences  such  that for each bit sequence $x$ there is a bit sequence  $y$ in  $G$   so that $\uul \rho (x \lra y) >p$. We prove within ZFC all the coincidences of cardinal characteristics that are the analogs of  the results above.
\end{abstract}

\maketitle 

\tableofcontents

%
%

\section{Introduction}
It is of  fundamental interest in computability theory to determine the inherent computational complexity of an  object, such as an infinite bit sequence, or more generally  a function~$f$ on the natural numbers. To determine  this  complexity,  one can place the object within  classes of objects  that  all have  a  similar complexity. Among such classes, we will focus on highness properties. They specify a sense in which  the object in question is computationally powerful.

The $\Gamma$-value of an infinite  bit sequence $A$,  introduced by  Andrews, Cai, Diamondstone, Jockusch and Lempp~\cite{andrews2013}, is a real in  between $0$ and $1$ that in a sense measures how well all oracle sets in its Turing degree can be approximated by computable sequences. For each $p\in (0,1]$, ``$\Gamma(A)<p$" is a highness property of $A$. The values $0, 1/2$ and $1$ occur \cite{hirschfeldt2016asymptotic,andrews2013}. Further, $\Gamma(A)> 1/2 \LR \Gamma(A) =1 \LR  A$ is computable \cite{hirschfeldt2016asymptotic}. Andrews et al.\  asked whether  the $\Gamma$-value can be strictly between $0$ and $1/2$. The precise definition of $\Gamma(A)$ will be given shortly  in Subsection~\ref{ss:Gamma}.

Monin   \cite{monin:asymptotic} answered their question in the negative, and also characterised the degrees with $\Gamma$-value $< 1/2$. He built  on some initial work of the present authors~\cite{Monin.Nies:15} involving    functions that agree with each computable function infinitely often.

Our goal is to  provide a systematic approach to the topic, relying on an analogy between highness properties of oracles   and    cardinal characteristics in set theory. In particular we   apply  methods analogous to the ones in \cite{monin:asymptotic} to cardinal characteristics.   

 Cardinal characteristics measure how far the set theoretic universe deviates from satisfying the continuum hypothesis. They are natural cardinals greater that $\aleph_0$ and at most $2^{\aleph_0}$.  We provide two examples.  For functions $f, g \colon \NN \to \NN$, we say that $g$ dominates $f$ if   $g(n) \ge f(n)$ for sufficiently large $n$.  The \emph{unbounding number} $\frb$ is the least size of a collection of functions $f$ such that no single function dominates   the entire collection. The \emph{domination  number} $\frd$ is the least size of a collection of functions   so that each function  is dominated by a function  in the collection. Clearly $\frb \le \frd$; in   appropriate models of set theory the inequality can be made strict, and   one can ensure that  $\frd< 2^{\aleph_0}$.  For an introduction to the topic see e.g.\ Blass~\cite{Blass:10}. A~general reference on cardinal characteristics  is the book~\cite{Bartoszynski.Judah:book}.

The analogy  between cardinal characteristics and highness properties of oracles in computability theory was first noticed and studied by Rupprecht~\cite{Rupprecht:thesis,Rupprecht:10}.   For instance, he observed  that   the analog of $\frb$  is the usual highness $A' \ge_T \ES''$ of an oracle $A$, and the analog of $\frd$  is being of hyperimmune degree.   Elaborating on Rupprecht's work,  Brooke-Taylor et al.~\cite{Brendle.Brooke.ea:14}  investigated  the analogy via  a    notation system that makes it possible to automatically  transfer many highness properties of oracles into cardinal characteristics, and vice versa.   

\smallskip

The rest of the introduction will provide more detail on the notions mentioned above, and describe the main results.

\subsection{Defining the $\Gamma$-value of a sequence}  \label[subsection]{ss:Gamma}
We recall how to define the $\Gamma$-value of  an infinite bit   sequence (often simply termed ``sequence");  this definition will only depend on its Turing degree.  For a  sequence  $Z$, also viewed as a subset of $\NN$,  the lower density is defined by
\bc $\underline \rho (Z)  = \liminf_n \frac{|Z \cap [0, n)|}{n}.$ \ec 
For    sequences $X,Y $ one  denotes  by $X\lra Y$ the sequence $Z$ such that $Z  (n) =1 $ iff   $ X(n) = Y(n)$.  
To measure how closely a   sequence $A$ can be approximated by a computable sequence $X$, Hirschfeldt et al.\  \cite{hirschfeldt2016asymptotic} defined \bc $\gamma(A)=\sup \{ \underline \rho(A\leftrightarrow X)\colon \, X \, \text{is computable}\}$.  \ec 
Clearly this depends on the particular sequence $A$, rather than its Turing complexity. Andrews et al.\ \cite{andrews2013}   were the first to study  the infimum of the $\gamma$-values over all $Y$ in the  Turing degree of $A$: 
$$\Gamma(A) = \inf \{ \gamma(Y) \colon  \, Y \equiv_T A \}.$$
Nies~\cite[Section 7]{Nies:17a} contains some   background on the $\Gamma$-value;  in particular, that  $1- \Gamma(A)$ can be seen as a Hausdorff pseudodistance between $\{Y \colon \, Y \leT A\}$ and the computable sets with respect to the Besicovitch distance $\overline \rho (U \triangle V)$ between bit sequences $U,V$ (where $\overline \rho$ is the upper density). Thus, a large value $\Gamma(A)$  literally means that $A$ is ``close to computable".

\smallskip

\subsection{Duality}
 
Cardinal characteristics often come in  pairs of dual cardinals.  
This duality stems from the way the characteristics are defined based on relations between suitable spaces. For instance, the unbounding number~$\frb$ is the dual of  the domination number~$\frd$.   The detail will be given in Definition~\ref{df: bd}. 

Brendle and Nies in~\cite[Section 7]{LogicBlog:15},   starting from  the  work in   \cite{hirschfeldt2016asymptotic,andrews2013}, considered  for $p \in [0,1/2]$ highness properties  $\+ D(p)$ such that   \begin{equation} \label{eqn:GammaD}  \Gamma(A)< p \RA A \in \+ D(p) \RA \Gamma(A)\le p.  \end{equation}  They defined   $\+ D(p)$   to be  the set of oracles $A$ that compute a bit sequence  $Y$ such that  
$\underline \rho( Y \lra X) \le p$ for each computable sequence $X$.  They then     applied    the framework  of Rupprecht~\cite{Rupprecht:thesis}, in the notation of Brooke-Taylor et al.\ \cite{Brendle.Brooke.ea:14}. This led to    cardinal characteristics $\frd(p)$,   the least size of a set $G$ of bit sequences  so that for each bit sequence $x$ there is a bit sequence  $y$ in  $G$   such  that $\uul \rho (x \lra y) >p$.  
  Dualising this both in computability and in set theory, they introduced  for each   $0\le p < 1/2$ the   highness property $\+ B(p)$, the class of oracles $A$ that compute a bit sequence $Y$  such that   for each computable  sequence $X$, we have $\underline \rho (X \lra Y) > p$, and the  analogous  cardinal characteristic  $\frb(p)$,  the least size of a set  $F$ of  bit sequences so that for each bit sequence $y$, there is a bit sequence $x$ in $F$  such that $\uul \rho (x \lra y) \le p$.   
  
 \subsection{Coincidences}  Extending Monin's methods \cite{Monin.Nies:15}, we will show that all the highness properties $\+ D(p)$ coincide for  $0< p < 1/2$, and similarly for the highness properties  $\+ B(p)$.   Since $\Gamma(A)< p \RA A \in \+ D(p)$, we re-obtain Monin's  result that $\Gamma(A) < 1/2$ implies $\Gamma(A) = 0$. Via analogous   methods within set theory, we show that $\mathrm{ZFC}$ proves  the coincidence of all the $\frd(p)$,  and of all the $\frb(p)$,  for  $0< p <  1/2$.  
  
  In Subsection~\ref{ss:density} we  will describe the coincidences in computability in  more detail. We first need to discuss some more concepts.

\subsection{Medvedev and  Muchnik reducibility} \label[subsection]{ss:MM} A  non-empty subset $\+ B$ of Baire space is  called a \emph{mass problem}.     A function $f \in \+ B$ is called   a \emph{solution} to the problem.  The easiest problem in this sense  is  the set of all functions. 

In this paper we will   phrase our highness properties in the language of  mass problems (rather than upward closed sets of Turing degrees as in \cite{Brendle.Brooke.ea:14}), and compare them via Medvedev and  Muchnik reducibility. The advantage of this approach is that we can keep track of potential uniformities when we give  reductions showing that one property is at least as computationally powerful as another.

 Let $\+ B$ and $  \+ C$ be mass problems. The reducibilities   provide two variants of saying  that any solution to $\+ C$ yields a solution to $\+ B$. The first, also called strong reducibility,  is the  uniform version:
 one writes   $\+ B \le_S \+ C$ (and says that $ \+ B$ is Medvedev reducible to $\+ C$) if there is a Turing functional $\Gamma$ with domain containing $\+ C$ such that $\fa g \in \+ C [  \Gamma^g \in \+ B]$.  Note that $\+ B \supseteq \+ C$ implies $\+ B \le_S \+ C$ via the identity functional.
 One writes    $\+ B \le_W \+ C$ (and says that $ \+ B$ is  weakly, or Muchnik reducible to $\+ C$) if $\fa g \in \+ C\ \ex f \in \+ B [ f \le_T g]$. Muchnik degrees correspond to end segments in the Turing degrees via sending $\+ C$ to the collection of oracles computing a member of $\+ C$.   In this way, viewing highness properties as mass problems and comparing them via   Muchnik reducibility $\le_W$,  is equivalent to viewing them as   end segments in the Turing degrees and comparing them via  reverse inclusion.

\smallskip

\subsection{A  pair of   dual mass problems for functions.}  \label[subsection]{ss:dualpair}
One can determine  the computational complexity of an object by comparing it to \emph{computable} objects of the same type.   
This  idea was used to introduce the density-related mass problems  $\+ D(p)$  and $\+ B(p)$. We will   apply  it    to introduce   two further    mass problems  of  importance in this paper.   We say that a function~$f$ is $\mathrm{IOE}$ if $\ex^\infty n \, [f(n) = r(n) ]$ for each computable function $r$. 
 We say that $f$ is $\mathrm{AED}$  if $\forall^\infty n \, [f(n) \neq r(n) ]$ for each computable function $r$. ($\mathrm{IOE}$ stands for ``infinitely often equal", while  $\mathrm{AED}$ stands for ``almost everywhere different".)

The study of the class  $\mathrm{AED}$ can be traced back to Jockusch~\cite[Thm.\ 7]{Jockusch:89}. He    actually  considered a stronger property of a function $f$ he denoted  by SDNR:   $\fa^\infty n \, [f(n) \neq r(n) ]$ for each \emph{partial}  computable function $r$.  Kjos-Hanssen, Merkle, and Stephan \cite[Thm.\ 5.1 $(1) \to (2)$]{Kjos.Merkle:11} showed that
 each non-high $\mathrm{AED}$ function  is SDNR.

The class  $\mathrm{IOE}$  was   introduced much later.   Kurtz \cite{Kurtz:81} showed that the mass problem of weakly  1-generic sets is Muchnik  equivalent to the functions not dominated by a computable function (the corresponding end segment consists of   the hyperimmune Turing degrees).  Using this fact,  it is not hard to show  that $\mathrm{IOE}$ is also Muchnik equivalent to  the class of  functions not dominated by a computable function.

An \emph{order function} $h$  is a non-decreasing, unbounded computable function. In computability theory, one often uses order functions as bounds to parameterise known classes of measuring computational  complexity. For instance, $\mathrm{DNC}(h)$ is the class of diagonally non-computable functions $f<h$. For another example, an oracle $A$ is $h$-traceable if each $A$-partial computable function has a c.e.\ trace of size bounded  by $h$ (see e.g.\ \cite[Ch.\ 8]{Nies:book}). 

We focus on  versions of the classes $\mathrm{IOE}$ and $\mathrm{AED}$ parameterised by an order function~$h$. By  $\mathrm{IOE}(h)$ we denote  the mass problem  of functions  $f$  such   $\ex^\infty n \, [f(n) = r(n) ]$ for each computable function $r< h$.   Dually,  $\mathrm{AED}(h)$  is the mass problem of functions   $f< h$ such that  $\forall^\infty n \, [f(n) \neq r(n) ]$ for each computable function $r$.

  Note that  $g\le h$ implies $\mathrm{IOE}(g) \supseteq \mathrm{IOE}(h)$ and $\mathrm{AED}(g) \subseteq \mathrm{AED}(h)$.  One can now ask the following: 
    \emph{For an  order function $h$ that grows sufficiently much  faster  than an order function $g$, do we obtain  $\mathrm{IOE}(g) <_W \mathrm{IOE}(h)$ and $\mathrm{AED}(g) >_W \mathrm{AED}(h)$?  }

    For the operator   $\mathrm{IOE}$,  separations for some  rather  special cases of functions $g,h$ were obtained in  \cite{Monin.Nies:15}. We  answer the full question for IOE  in the affirmative. Theorem~\ref{th-ioe}, which is  joint work with Joseph S.\ Miller that will be  included here, roughly speaking states that $h$ needs to be growing faster than $2^{g(2^{n\cdot n})}$ for a separation.

For the operator $\mathrm{AED}$,  the answer was  known already. Recent work of Khan and Miller \cite{Khan.Miller:17} provides  a hierarchy for  the mass problems of low $\mathrm{DNR}(h)$ functions.   Khan and Nies \cite[Section~2]{LogicBlog:17}   turned    these mass problems into  mass problems $\mathrm{AED}(\wt h)$ for $\wt h$ close to  $h$, preserving weak reducibility. 

The characteristics $\frb(\neq^*,  h)$    are  analogous to the mass problems $\mathrm{AED}(h)$;   detail will be given in  Section~\ref{s:CC}.  Kamo and Osuga~\cite{Kamo.Osuga:14} have proved  that it is  consistent with ZFC  to have distinct  cardinal characteristics  $\frb(\neq^*,  h)$ depending on the growth of the function $h$. A similar result is unknown at present for the dual characteristics   $\frd(\neq^*,  h)$.

\smallskip
\subsection{Density}  \label[subsection]{ss:density} With the reducibilities  discussed in Subsection~\ref{ss:MM} in mind, the above mentioned highness properties related to $\Gamma$, introduced by Brendle and Nies in \cite[Section 7]{LogicBlog:15},   will now be considered as  mass problems. They consist of  $\{0,1\}$-valued functions on $\NN$, i.e.,  infinite bit sequences.    Let $p$ be a real with $0 \le p< 1$.     
  $\+ D(p)$ is    the set of bit sequences $y$  such that  $\uul \rho (x \lra y) \le  p$  for each computable  set $x$.  Note that this   resembles the definition of $\mathrm{IOE}$. 
   $\+ B(p)$    is the set of bit sequences $y$  such that   $\uul \rho (x \lra y) > p$ for each computable  set $x$.  This   resembles the definition of $\mathrm{AED}$.

Clearly $0 \le p<q < 1$ implies $\+ D(p) \sub \+ D(q)$ and $\+ B(p) \supseteq  \+ B(q)$. Our first result,  \Cref{thm:mainDelta},  shows  that  there actually is no proper hierarchy in the Muchnik degrees when the parameter is positive. It also provides  a   characterisation by  a combinatorial class,   relying on  agreement of functions with computable functions,  rather than on density:   
 \bc       $\+ D(p)\equiv_W \mathrm{IOE}( \tp{(\tp n)})$ and $\+ B(p)\equiv_S \mathrm{AED}( \tp{(\tp n)})$ \ec 
 for arbitrary   $p\in (0,1/2)$.   The corresponding result for cardinal characteristics is \Cref{th:sht} below.  The outer  exponential function in the bound simply stems from the fact that we view function values as encoded by binary numbers, which correspond to blocks in the   bit sequences: if a bound~$h$ has the form $\tp{\hat h}$ for an order function $\hat h$, then a function $f < h$ naturally corresponds to a bit sequence which is the concatenation of blocks of length $\hat h(i)$ for $i \in \NN$.

 As part of the proof of  \Cref{thm:mainDelta}, we  show  in  a lemma that the parameterised classes $\mathrm{IOE}(h)$ and $\mathrm{AED}(h)$ don't depend too sensitively on the bound~$h$: if $g(n) = h(2n)$ then $\mathrm{IOE}(g) \equiv_W \mathrm{IOE}(h)$ and $\mathrm{AED}(g) \equiv_S \mathrm{AED}(h)$.   Since the first equivalence we obtain   is merely  Muchnik,  in \Cref{thm:mainDelta}  we also only have Muchnik in its first equivalence.  Note that by the lemma, in the above, we can replace $\mathrm{IOE}( \tp{(\tp n)})$ by $\mathrm{IOE}( \tp{(\tp {n\cdot r})})$ for any $r>0$.
 
We remark that  after the  first version of this paper was posted on arXiv in December 2017~ \cite{Monin.Nies:17}, as part of a large study, Greenberg, Kuyper and Turetsky~\cite{Greenberg.etal:18} sketched  their own version of our   coincidence  results for highness classes and for cardinal characteristics. Based on the work of Rupprecht, they introduced  a systematic machinery that also applies to reductions. It now  suffices to state a single abstract theorem \cite[Thm.\ 6.15]{Greenberg.etal:18}, which implies all four coincidences.    
 
 We think that each approach has its advantages.  Ours is more direct and requires the reader to assimilate less general theory before proceeding to the result; it also deals in a concrete way with the non-uniformities in the computability case, for instance in Lemma~\ref{lem:doubleR}. Their approach is general, and hence  superior to ours  towards understanding the reason for  the persistent analogies between set theory and computability, and the dualities within each area. Also, once the general machinery is available,  it saves work in proving particular coincidences.

 \subsection*{Acknowledgements} As mentioned, several  of the questions studied here arose  in  work of  J\"org  Brendle and the second author that has been  archived in~\cite[Section 7]{LogicBlog:15}, where the cardinal characteristics relating to density were introduced. We thank Brendle for these very helpful discussions. 
 We thank Joseph Miller for his  contribution   towards  Section 5 in this paper. 
 Nies was  supported
in part by the Marsden Fund of the Royal Society of New Zealand, UoA
13-184 and UoA 19-346. 
The work was completed while the authors visited the Institute for Mathematical Sciences at NUS during the  2017 programme ``Aspects of Computation".

\section{Defining mass problems based on relations}
Towards proving our   main theorems, we will need a general formalism to define mass problems based on relations, similar to~\cite{Brendle.Brooke.ea:14}.  We consider ``spaces" $X,Y$, which will be effectively closed subsets of  Baire space. Let the variable $x$ range over $X$, and let $y$ range over $Y$. Let  $R \subseteq  X \times Y$ be a relation, and  let $S = \{  \langle y,x \rangle \in Y \times X \colon \neg xR y\}$.

 \begin{definition} \label[definition]{df:DB} We define the  pair of dual mass problems 
\begin{eqnarray*} \+ B(R) &=&   \{ y\in Y \colon  \fa x \ \text{computable} \ [xRy]     \} \\
 \+ D(R) =  \+ B(S) &=&   \{ x \in X \colon  \fa y \ \text{computable} \ [\neg xRy]   \}.  \end{eqnarray*} \end{definition}
To re-obtain the mass problems discussed in the introduction, we consider  the following two types of relation.

\begin{definition} \label[definition]{df:relations}
\n {\it 1.}   Let $h \colon \omega \to \omega - \{0,1\} $. Define  for $  x \in {}^\omega\omega$ and 

\n $  y \in \prod_n \{0, \ldots, h(n)-1\} \sub  {}^\omega\omega$,   
\[ x \neq^*_h  y \LR \fa^\infty n \,  [ x(n) \neq y(n)]. \]

\n {\it 2.} Recall that  $\uul \rho (z) = \liminf_n |z \cap n|/n$ for a bit sequence $z$. 
 Let $ 0 \le p < 1$. Define,   for $x,y \in {}^\omega 2$
\[ x \bowtie_p y \LR   \underline \rho (x \lra y)  >p, \]
where $x \lra y$ is  the set of $n$ such that $x(n) = y(n)$. \end{definition}
For the convenience of the reader we summarise the specific  mass problems determined by these relations.  

\begin{remark}  Let $h$ be a computable function. Let $p$ be a real with ${0 \le p \le 1/2}$. 

\medskip

\n   \fbox{$\mathrm{IOE }(h)$}   is our     notation for $\+ D(\neq^*_h)$,   the set of functions $y$ such that for each computable function $x < h$, we have $\ex^\infty n \, x(n) = y(n)$.

\medskip

\n  \fbox{$\mathrm{AED }(h)$} is our notation for $\+ B(\neq^*_h)$, the set of functions $y <  h$  such that for each computable function $x$, we have     $\fa^\infty n \, x(n) \neq y(n)$.   

\medskip

\n \fbox{$\+ D(p)$} is our  short notation for   $\+ D(\bowtie_p)$,    the set of bit sequences $y$  such that for each computable set $x$, we have $\uul \rho (x \lra y) \leq p$.  

\medskip

\n     \fbox{$\+ B(p)$} is our  short notation for   $\+ B(\bowtie_p)$, the set of bit sequences $y$  such that   for each computable  set $x$, we have $\uul \rho (x \lra y) > p$.
\end{remark}

\section{Coincidences of Muchnik degrees} \label[section]{s:Nies_Delta}

As mentioned, our goal is to  show   \bc       $\+ D(p) \equiv_W \mathrm{IOE}(\tp{(\tp n)})$ and $\+ B(p) \equiv_S \mathrm{AED}( \tp{(\tp n)})$ \ec for arbitrary   $p\in (0,1/2)$.
 We begin with some preliminary facts of independent interest.
 On occasion we denote a function $\lambda n. f(n)$ simply by $f(n)$.

 \begin{lemma} \label[lemma]{lem:doubleR}  \bi \item[(i)] Let $h$ be nondecreasing and $g(n) = h(2n)$. We have  
 
 \n  $\mathrm{IOE}(h) \equiv_W\mathrm{IOE}(g)$  and $\mathrm{AED}(h) \equiv_S \mathrm{AED}(g)$.

\item [(ii)] For each $a,b >1$ we have   

\n $\mathrm{IOE}({\tp{(a^n)})}\equiv_W\mathrm{IOE}({\tp{(b^n)}})$ and   $\mathrm{AED}({\tp{(a^n)})}\equiv_S \mathrm{AED}({\tp{(b^n)}})$. \ei \end{lemma}
Note that the duality appears to be incomplete: for the statement involving the $\mathrm{IOE}$-type problems, we   only obtain weak equivalence. We ignore at present whether strong equivalence holds. 

\begin{proof} (i) Trivially, $h \le g$ implies    $\mathrm{IOE}(h) \supseteq\mathrm{IOE}(g)$   and $\mathrm{AED}(h) \sub \mathrm{AED}(g)$.   
 So it suffices to provide  only one reduction in each case.

\n  \fbox{$\mathrm{IOE}(h) \ge_W\mathrm{IOE}(g)$:} Let $y < h$ be a function in $\mathrm{IOE}( h)$. Let $\hat y_1 < h(2n)$ and $\hat y_2 < h(2n+1)$ be defined by $\hat y_1(n) = y(2n)$ and $\hat y_2(n) = y(2n+1)$. We claim that at least one function among  $\hat y_1, \hat y_2$   belongs to $\mathrm{IOE}( g)$. Suppose otherwise. Then there are  computable functions $x_1,x_2 < g$ which differ almost all the time from $\hat y_1$ and $\hat y_2$, respectively. Since $h$ is nondecreasing, the computable function $x$ defined by $x(2n) = x_1(n)$ and $x(2n+1) = x_2(n)$ satisfies $x < h$. It is clear that $x$ differs almost all the time from $y$, which contradicts   $y \in\mathrm{IOE}(h)$.

\n  \fbox{$\mathrm{AED}(h) \le_S \mathrm{AED}(g)$:} 
Let   $y < g$ be a function in $ \mathrm{AED}( g)$. Let $\hat y(2n+i) = y(n)$ for $i \le 1$, so that $\hat y < h$. Given any computable function $x$, for almost every $n$ we have $x(2n) \neq y(n) $ and $x(2n+1) \neq y(n)$.  Therefore $ x(n)  \neq \hat y(n)$ for almost every $n$. Hence $\hat y \in \mathrm{AED}(h)$.

(ii) is immediate from (i) by iteration,  using that   $a^{2^i} >b$  and $b^{2^i} >a$ for sufficiently large $i$. 
\end{proof}
%
The following operators will be used for the rest of the section.
  \begin{definition}[\bf The operators $L_h$ and $K_h$] \label[definition]{df:KhLh} Let $h$ be a function of the form $2^{\hat h}$ with $\hat h \colon \, \omega \to \omega$, and let $X_h$ be the space of all $h$-bounded functions. 
 For  such a    function   we    view $x(n)$ either as a number, or as a binary string   of length $\hat h(n)$ via the binary expansion with leading zeros allowed. We  define $L_h\colon X_h \to   {}^\omega 2$ by   $L_h(x) = \prod_n x(n) $, i.e.  the concatenation of these strings.  We let $K_h \colon  {}^\omega 2 \to X_h$ be the inverse of $L_h$. \end{definition}

\begin{lemma} Let $a \in \omega - \{0\}$. We have  
\bc $\mathrm{IOE}({\tp{(a^n)}}) \ge_S  \+ D(1/a)$ and $\mathrm{AED}({\tp{(a^n)}}) \le_S  \+ B(1/a)$.  \ec
\end{lemma}

\begin{proof} 
Let $I_m$ for $m \ge 2$  be the $(m-1)$-th consecutive interval of length $a^m$ in $\omega-\{0\}$, i.e.
\[ I_m = \left [\frac{a^m-1}{a-1} , \frac{a^{m+1} -1 } {a-1}\right) \]
Let $h(m) = \tp{(a^m)}$. Let us first show that $\mathrm{IOE}({\tp{(a^n)}}) \ge_S \+ D(1/a)$. Let $y < \tp{(a^n)}$ be a function in $\mathrm{IOE}( {\tp{(a^n)}})$ and let  $\hat y = L_h(y)$. Given a computable set $x$, let $x'=K_h(1-x)$ (where $1-x$ is the complement of $x$). As $x'(m)=y(m)$ for infinitely many $m$,   for infinitely many intervals $m$,   all bits of $x$ with location in $I_m$    differ from all the bits of $\hat y$ in this location. It follows that $\hat y \in \+ D(1/a)$.

Let us now show that $\mathrm{AED}({\tp{(a^n)}}) \le_S \+ B(1/a)$. Let $y \in \+ B(1/a)$, and let $\hat y = K_h(y)$. Given a computable function $x < h$, let $x' = 1- L_h(x)$. Since $\uul \rho( x' \lra y) > 1/a$, for large enough $n$, there is $k \in I_n$ such that $x'(k) = y(k)$. Hence we cannot have $x(n) = \hat y(n)$.  Thus $\hat y \in \mathrm{AED}({\tp{(a^n)}})$.
\end{proof}

\begin{remark} \label[remark]{rem:D0} Let $\hat h$ be an order  function  such that $\fa a\, \fa^\infty m \,  \hat h(m) \ge a^m$. An argument similar to the one in the foregoing proof shows that    \bc $\mathrm{IOE}({\tp{\hat h(m)}}) \ge_S  \+ D(0)$ and  $\mathrm{AED}({\tp{\hat h(m)}}) \le_S  \+ B(0)$. \ec     In this case one chooses the $m$-th interval of  length $\hat h (m)$. \end{remark}

\begin{thm} \label[theorem]{thm:mainDelta} Fix any  $p\in (0,1/2)$. We have \bc    $\+ D(p)\equiv_W\mathrm{IOE}(\tp{(\tp n)})$ and  $\+ B(p)\equiv_S \mathrm{AED}( \tp{(\tp n)})$. \ec
\end{thm}

The rest of the section is dedicated to the proof of \Cref{thm:mainDelta}. The two foregoing lemmas imply    $\+ D(p)\le_W \mathrm{IOE}(\tp{(\tp n)})$ and  $\+ B(p)\ge_S  \mathrm{AED}( \tp{(\tp n)})$. It remains to show the more difficult converse reductions    $\+ D(p)\ge_W  \mathrm{IOE}(\tp{(\tp n)})$ and $\+ B(p)\le_S   \mathrm{AED}( \tp{(\tp n)})$. Let us informally describe the proof of the first reduction, which is     based on arguments in   Monin's proof   \cite{monin:asymptotic} that $\Gamma(A)< 1/2 \LR \Gamma(A) =0$ for each $A \sub \omega$.  

%
Given $A \in \+ D(p)$ we want to find a function $ f\leT A$ that agrees with each computable function $g < 2^{(2^n)}$  infinitely often.  For an appropriate $k$ let $\hat h(n) = \lfloor \tp{n/k} \rfloor$ and $h(n) = \tp{ \hat h(n)}$. We split the bits of $A$ into consecutive intervals of length $\hat h(n)$. The first step  (\Cref{cl:AR}) makes the crucial transition from the density setting towards  the   setting of functions agreeing on certain arguments. We will   show that for $k$ large enough, the function $f_0=K_h(A)< h$ has the property that for each computable function $g< h$,  for infinitely many $n$, $f_0(n)$ and $ g(n) $ disagree on a fraction of fewer than $p$ bits when  viewed as binary strings of length $\hat h(n)$. 

In the second step (\Cref{cl:BR}) we    use $f_0$ to compute a special kind of approximation $s$ to computable functions: for each $n$, $s(n)$ is a set of $L$ many  values (where $L$ is an appropriate constant) such that for every computable function $g < h$ we have $\exists^\infty n\ g(n) \in s(n)$. Such a  function $s$ will be  called a slalom (another term in use is  ``trace"); we also say that $s$ \emph{captures}~$g$. This important step uses a result from the theory of error-correcting codes, which determines the constant $L$. 
 
In the third step (\Cref{lem:doubleR2}), which is non-uniform,  we replace $s$ by a slalom $s'$ such that still $s'(n)$ has size at most $L$, but now all computable functions $ g$ with $g(n) < 2^{(\tp{Ln})}$ are captured infinitely often. 

 In a final, non-uniform step   (\Cref{cl:finalR}) we then  compute from  $s'$   a function $f$ as required; for some $i$, $f(n)$ is the $i$-th block of length $2^n$ of the $i$-th element of $s(n)$. 

%
We now provide the detailed argument. For this recall Definition~\ref{df:DB}.

\begin{definition} \label[definition]{df: A} For strings $x,y$ of length $r$, the normalized Hamming distance is defined as the proportion of bits on which $x,y$ disagree, that is, 
\[ d(x,y) = \frac 1 r |\{ i \colon x(n) (i)    \neq y(n)(i) \}| \] \end{definition}
\begin{definition}  \label[definition]{df: B} Let $h$ be a function of the form $2^{\hat h}$ with $\hat h \colon \, \omega \to \omega$, and let $X_h$ be the space of $h$-bounded functions.  Let $q \in (0, 1/2)$. 
 We   define a relation on $X_h \times X_h$ by: 
    \[x \neq^*_{\hat h,q} y \LR \fa^\infty n  \, [ d(x(n), y(n)) \ge q]   \]   
 namely for almost every $n$ the strings $x(n)$ and $y(n)$  disagree on a proportion of at least  $q$ of the bits. We will usually write $\la \neq^*, \hat h,q \ra $ for this  relation.  
\end{definition}

\begin{claim} \label[claim]{cl:AR} Let $q \in (0, 1/2)$. For each $c\in \omega$ such that $\frac 2 c<q$,  there is $k \in \omega$ such that 
     \begin{eqnarray*} \+ D(q- \frac 2 c) &\ge_S& \+ D \la \neq^*, \lfloor \tp{n/k} \rfloor, q\ra \text{ and } \\ \+ B(q- \frac 2 c) &\le_S & \+ B \la \neq^*, \lfloor \tp{n/k} \rfloor, q\ra.      \end{eqnarray*} 
\end{claim}
\begin{proof} \NumberQED{cl:AR}
 Let $k $ be large enough so that $\aaa -1 < \frac 1 {2c}$ where  $\aaa = \tp{1/k}$. Let  $\hat h(n) = \lfloor\aaa^n \rfloor$ and  $h = \tp{\hat h}$. Write $H(n) = \sum_{r\le  n} \hat h(r)$. By the usual  formula for  the geometric series,  $$\sum_{r\le n} \aaa^r  = \frac {\aaa^{n+1} -1}  {\aaa -1} \le H(n)+ n+1$$ 
 and therefore $\aaa^{n+1} -1 \le  \frac 1 {2c} (H(n)+n+1)$. 
  If  $n$ is sufficiently large so that  $H(n) \ge n+1 +2c$, we now have
\begin{equation}
 \hat h(n+1) \le \frac 1 c H(n). \label{eq:eq1}
\end{equation}

  To prove the claim we also rely on the following.
\begin{fact} Let $x,y  < h$ be functions such that $\fa^\infty n \, [d(x(n), y(n) ) \leq 1 - q]$. Then $\underline \rho (L_h(x) \lra L_h(y))  > q- \frac 2c$. \end{fact}

 To see this, note that by hypothesis, for almost every $n$ we have that $L_h(y) \rest {H(n)}$ agrees with $L_h(x) \rest {H(n)}$ on a fraction of at least $q$ bits. For any $n$ and any $m$ with $H(n) \leq m \leq {H(n+1)}$, we   have that $L_h(y) \rest m$ agrees with $L_h(x) \rest m$ on a fraction of at least $\frac{H(n) q}{H(n) + \hat h(n+1)}$ bits, which is by (\ref{eq:eq1}) a fraction of at least $\frac{H(n) q}{H(n) + (1/c) H(n)}$ bits. It follows that for almost every $m$, we have that $L_h(y) \rest m$ agrees with $L_h(x) \rest m$ on a fraction of at least $\frac{q}{1 + 1/c}> q - \frac 2c$ bits. This  implies in particular that $\underline \rho (L_h(x) \lra L_h(y))  > q- \frac 2c$. The fact is proved.\\

Firstly we  show that $\+ D(q- \frac 2c) \ge_S \+ D \la \neq^*, \lfloor \tp{n/k} \rfloor, q\ra$. Let $y \in \+ D(q- \frac 2c)$. Let $y' = K_h(y)$. By the fact above,  there is no  computable function $x < h$ such that $\fa^\infty n \, [d(x(n), y'(n) ) \leq 1 - q]$, as otherwise we would have $L_h(y') = y \notin \+ D(q- \frac 2c)$ which is a contradiction.    
Therefore, for every computable function $x < h$ we have $\exists^\infty n \, [d(x(n), y'(n) ) > 1 - q]$. Now let $x < h$ be a computable function and let $x' = K_h(1 - L_h(x))$. As $x' < h$ is computable we must have $\exists^\infty  n \, [d(x'(n), y'(n) ) > 1 - q]$. But then we also have $\exists^\infty  n \, [d(x(n), y'(n) ) < q]$. As this is true for any computable function $x < h$ we then have $y' \in \+ D \la \neq^*, \lfloor \tp{n/k} \rfloor, q\ra$. 
  

Secondly we  show that $\+ B(q- 2/c) \le_S \+ B \la \neq^*, \lfloor \tp{n/k} \rfloor, q\ra$.

\n
Let $ y \in \+ B \la \neq^*, \lfloor \tp{n/k} \rfloor, q\ra$. Thus, $y< h$ and  $\fa^\infty n \, [d(x(n), y(n) \ge q]$ for each computable function $x < h$. Let $y' = K_h(1  -  L_h(y))$. Then \bc $\fa^\infty n \, [d(x(n), y'(n) ) \le  1 - q]$ \ec for each computable function $x < h$. By the fact above, we then have that $\underline \rho (L_h(x) \lra L_h(y'))  > q- 2/c$ for each computable function $x < h$. It follows that $L_h(y') \in B(q- 2/c)$. 
\end{proof}
 
For $L \in \omega$, an \emph{$L$-slalom} (also called trace) is a function $s\colon \, \omega \to \omega^{[\le L]}$, i.e.\ a function that  maps  natural numbers to sets  of natural numbers with a  size of at most $L$. 
\begin{definition}  \label[definition]{df: trace} Fix  a function $u \colon \omega \to \omega$ and $L \in \omega$. Let $X$ be the space of $L$-slaloms (or  traces)   $s$ such that $\max s(n) < u(n)$ for each $n$. Thus $s$ maps natural numbers to sets of natural numbers of size at most $L$, represented by strong indices. Let $Y$  be the set of functions such that  $y(n)  <  u (n)$ for each $n$. Define a relation on $X \times Y$ by
 \[s \not \ni^*_{u,L} y \LR \fa^\infty n [s(n) \not \ni y(n)]. \] 
  We will   write $\la \not \ni^*, u ,L \ra $ for this  relation. 
\end{definition} 

For what follows, we use the list decoding capacity theorem from the theory of error-correcting codes due to  Elias~\cite{elias1991}.  Given $q $ as above  and $L \in \omega$, for each $r$ there is a ``fairly large" set $C$ of strings of length $r$ (the  allowed code words) such that for each string, at most $L$ strings in $C$ have normalized Hamming distance less than $q$ from $\sss$. Intuitively speaking,  there is only a small set of strings that could be the error-corrected  version of $\sss$.

Given a string $\sss$ of length $r$, let $B_q(\sss)$ denote  an  open  ball  of radius $q$ around $\sss$ in the normalized Hamming distance, namely, \bc $B_q(\sss) = \{ \tau \in  {}^r 2 \colon \, \sigma, \tau \text{  disagree on fewer than $qr$ bits}\}$. \ec
 \begin{theorem}[List decoding, Elias \cite{elias1991}] \label[theorem]{lem:ListDec} Let $q\in (0,1/2)$.  There are $\epsilon >0$ and $L \in \omega$ such that for each $r$, there is a set $C$ of $\tp{\lfloor \epsilon r \rfloor }$ strings of length $r$ as follows: \bc $\forall \sss \in {}^r 2 \, [ | B_q(\sss) \cap C | \le L]$.  \ec
 \end{theorem}

The previous theorem allows us to show the following:

\begin{claim} \label[claim]{cl:BR} Given $q < 1/2$, let $L, \epsilon $ be as in \Cref{lem:ListDec}. Fix  a nondecreasing  computable function $\hat h$, and   let  $u(n)= \tp{\lfloor \epsilon \hat h(n) \rfloor }$. We have 
\bc $\+ D \la \neq^*,  \hat h , q\ra \ge_S \+ D \la \not \ni^*, u,L \ra$ and $  \+ B \la \neq^*,  \hat h , q\ra \leq_S  \+ B \la \not \ni^*, u,L \ra$.  \ec 
\end{claim}
\begin{proof} \NumberQED{cl:BR}
%
Given a number  $r$ of the form $\hat h(n)$, one can  compute a set $C= C_r$ as in \Cref{lem:ListDec}. Since $|C_r| = \tp{\lfloor \epsilon r \rfloor}$  there is  a uniformly computable sequence $\seq {\sss^r_i}_{i<  \tp{\lfloor \epsilon r \rfloor}}$ listing $C_r$ in increasing lexicographical order. 
 
%


Firstly we show that $\+ D \la \neq^*,  \hat h , q\ra \ge_S \+ D \la \not \ni^*, u,L \ra$. Suppose that $y \in \+ D \la \neq^*,  \hat h , q\ra$. Let $s$ be the uniformly $y$-computable $L$-slalom such that
\bc $s(n) = \{i < u(n) \colon \, d(\sss^{\hat h(n)}_i, y(n) ) < q\}$. \ec 
Let now $x < u$ be a computable function. Since $d(\sss^{\hat h(n)}_{x(n)}, y(n) ) < q$ for  infinitely many $n$,  for  infinitely many $n$  we have $x(n) \in s(n)$. It follows that $s \in  {\+ D \la \not \ni^*, u,L \ra}$, as required.

Secondly we  show that $\+ B \la \not \ni^*, u,L \ra \geq_S \+ B \la \neq^*,  \hat h , q\ra$. Suppose that $y \in {\+ B \la \not \ni^*, u,L \ra}$.  Let  $h = 2^{\hat h}$,  and let $\wt y < h $ be the function given by $\wt y(n)= \sss^{\hat h(n)}_{y(n)}$. We show that $\wt y \in \+ B \la \neq^*,  \hat h , q\ra$. Let $x < h$ be a computable function. Let %
 \bc  $s_x(n) = \{i < u(n) \colon \, d(\sss^{\hat h(n)}_i, x(n) )< q\}$. \ec 
Note that $\seq{s_x(n)}\sN n$ is an $L$-slalom because the listing $\seq {\sss^r_i}_{i<  \tp{\lfloor \epsilon r \rfloor}}$ has no repetitions. Since $y \in \+ B \la \not \ni^*, u,L \ra$, for almost every $n$ we have $y(n) \not \in s_x(n)$. Hence also for almost every $n$ we have $d(\wt y(n),x(n) \ge q$, as required.
\end{proof} 

We next need an amplification tool   in the context of slaloms. The proof is almost verbatim the one in \Cref{lem:doubleR}(i), so we omit it.

\begin{claim} \label[claim]{lem:doubleR2} Let $ L \in \omega$, let the  computable function $u$ be nondecreasing and let $w(n) = u(2n)$. We have   \bc $\+ D \la \not \ni^*, u,L \ra \equiv_W \+ D \la \not \ni^*, w,L \ra$  and $\+ B \la \not \ni^*,u, L \ra \equiv_S \+ B\la \not \ni^*,w,L\ra$. \ec
 \end{claim}
Iterating the claim,  starting with the function $\hat h(n) = \lfloor \tp{n/k}\rfloor$ with $k$ as in \Cref{cl:AR},  we   obtain that   $\+ D\la \not \ni^*,\tp{\hat h}, L \ra \equiv_W \+ D \la \not \ni^*,\tp{(L 2^n)}, L \ra$  and   $\+ B\la \not \ni^*,\tp{\hat h}, L \ra \equiv_S \+ B \la \not \ni^*,\tp{(L 2^n)}, L \ra$. 
It remains to verify the following, which would work for any computable function $\hat h(n)$ in place of the  $2^n$ in the exponents.
\begin{claim} \label[claim]{cl:finalR} \
 \bc $\+ D \la \not \ni^*,\tp{(L 2^n)}, L \ra \geq_W\mathrm{IOE}(\tp{(\tp n)})$  and  $  \+ B \la \not \ni^*,\tp{(L 2^n)}, L \ra  \leq_S \mathrm{AED}( \tp{(\tp n)})$. \ec
\end{claim}
\begin{proof} \NumberQED{cl:finalR}
 Given  $n$, we write  a number $k< \tp{(L \tp n)}$ in binary with leading zeros if necessary, and so can view $k$ as a binary string of length $L \tp n$. We view such a string as    consisting  of $L$ consecutive blocks of length  $\tp n$.

Firstly we show $\+ D \la \not \ni^*,\tp{(L 2^n)}, L \ra \geq_W\mathrm{IOE}(\tp{(\tp n)})$. Let $s \in \+ D \la \not \ni^*,\tp{(L 2^n)}, L \ra$. For every $i < L$, let $y_i$ be the $s$-computable function such that $y_i(n)$ is the $i$-th block of the $i$-th element of $s(n)$. Suppose for a contradiction that for every $i$ we have a computable function $x_i \leq 2^{(2^n)}$ which differs on almost every argument from $y_i$. Let $x \leq 2^{(L2^n)}$ be the computable function defined by $x(n)$ to be the concatenation of $x_i(n)$ for each $i \leq L$. Then also for almost every $n$ we have $s(n) \not \ni x(n)$, which is a contradiction. Therefore we must have that $y_i \in\mathrm{IOE}(\tp{(\tp n)})$ for some $i \leq L$.
 
Secondly we  show $\mathrm{AED}( \tp{(\tp n)}) \geq_S \+ B \la \not \ni^*,\tp{(L 2^n)}, L \ra$. Let $y \in \mathrm{AED}( \tp{(\tp n)})$. That is, $y< 2^{(2^n)}$ and  $\fa^\infty n \, x(n) \neq y(n)$ for each computable function $x$.    Let $y'$ be the function bounded by $2^{(L2^n)}$ such that for each $n$, each block of $y'(n)$ equals $y(n)$.  
Given a computable $L$-slalom $s$ with $\max s(n) <  \tp{(L 2^n)}$, for $i< L$ let $x_i$ be the computable function  such that $x_i(n) $ is the $i$-th block of the $i$-th element of $s(n)$ (as before we may assume that each string in $s(n)$ has length $L2^n$). For sufficiently large $n$,  we have for all $i < L$ that $y(n) \neq x_i(n)$.  Hence $\fa^\infty n \, s(n) \not \ni y'(n)$ and thus $y' \in \+ B \la \not \ni^*,\tp{(L 2^n)} \ra$.
\end{proof}

Using the results above, we can now       finish the arguments   that   $\+ D(p) \geq_W \mathrm{IOE}(\tp{(\tp n)})$ and  $\mathrm{AED}( \tp{(\tp n)}) \geq_S \+ B(p)$. 

\begin{proof}[Proof of \Cref{thm:mainDelta}, completed] \NumberQED{thm:mainDelta} \ 

\n Pick $c$ large enough such that 
 ${q= p+\frac 2c} < 1/2$.  
 By \Cref{cl:AR} there is $k$ such that 
\bc $\+ D(p) \geq_S \+ D \la \neq^*, \lfloor \tp{n/k} \rfloor, q\ra$ and  $\+ B \la \neq^*, \lfloor \tp{n/k} \rfloor, q\ra \geq_S \+ B(p)$. \ec

\n  By \Cref{cl:BR} there are $L\in \omega$ and  $\epsilon>0$ such that where $\hat h(n) =  \lfloor \tp{n/k} \rfloor$ and $u(n)= \tp{\lfloor \epsilon \hat h(n) \rfloor }$,  
\bc $\+ D \la \neq^*,  \hat h , q\ra \ge_S \+ D \la \not \ni^*, u,L \ra$ and $ \+ B \la \not \ni^*, u,L \ra \geq_S \+ B \la \neq^*,  \hat h , q\ra$. \ec 


\n Applying \Cref{lem:doubleR2}  sufficiently often we have

\bc $\+ D \la \not \ni^*, u,L \ra \equiv_W \+ D \la \not \ni^*, \tp{(L 2^n)}, L \ra$ and $\+ B \la \not \ni^*,u, L \ra \equiv_S \+ B\la \not \ni^*, \tp{(L 2^n)}, L\ra$. \ec


\n Finally

\bc $\+ D \la \not \ni^*,\tp{(L 2^n)}, L \ra \geq_W\mathrm{IOE}(\tp{(\tp n)})$ and $\mathrm{AED}( \tp{(\tp n)}) \geq_S \+ B \la \not \ni^*,\tp{(L 2^n)}, L \ra$.\ec
by \Cref{cl:finalR}. Combining all this  yields the theorem.
\end{proof}
%
 \subsection*{The   $\Delta$-value of a Turing oracle} \label[subsection]{s:GD}
  The   $\Gamma$-value of a Turing oracle $A$ was   defined in \Cref{ss:Gamma} of the introduction. It is closely connected to the classes $\+ D(p)$ as in Eqn.\  \ref{eqn:GammaD} above. (Recall that we now view these classes as mass problems, so in   Eqn.\ \ref{eqn:GammaD}  we replace $A \in \+ D(p)$ there by  $\ex Y \le_T A \,  [ Y \in \+ D(p)] $.) Its dual   was  first considered by Merkle, Stephan and the second author \cite[Part 3]{LogicBlog:16}, and then in \cite[Section 7]{Nies:17a}. 
\begin{definition}  \label[definition]{def:delta} Let \begin{eqnarray*} \delta(A)&=&\inf\{ \underline \rho(A\leftrightarrow X)\colon \, X \, \text{computable} \}\\ \Delta(A) & =& \sup \{ \delta(Y) \colon  \, Y \le_T A \}.\end{eqnarray*}
\end{definition}
Intuitively,   $\Gamma(A)$ measures  how well computable sets can approximate    the sets that $A$ computes, counting   the asymptotically worst case (the infimum over all $Y \le_T A$). In contrast,   $\Delta(A)$ measures how well the sets   that $A$ computes can approximate   the computable sets, counting  the asymptotically best case (the supremum over all $Y \le_T A$).  Clearly $\Delta(A) \le 1/2$ for each $A$. 
 \begin{cor}\label{cor:gamdel} (i) $\Gamma(A) < 1/2$ implies $\Gamma(A) =0$. 
 
 \n (ii) $\Delta(A)>0 $ implies $\Delta(A) = 1/2$.
 \end{cor}
 
 \begin{proof} By the definitions, for each $p \in  (0, 1/2)$, we have \bc $\Gamma(A)< p \RA  \ex Y \le_T A \,   [Y \in \+ D(p)]  \RA \Gamma(A)\le p$    \ec and dually \bc $\Delta(A)>  p \RA  \ex Y \leT A \, [Y  \in \+ B(p)] \RA \Delta(A)\ge p$. \ec Now apply \Cref{thm:mainDelta}. \end{proof}

The $\Delta$-values $0$ and $1/2$ can be realized by the following two facts already mentioned in \cite[Part 3]{LogicBlog:16}.
\begin{prop} Let $A$ compute a Schnorr random $Y$. Then $\Delta(A) = 1/2$. \end{prop}
\begin{proof} \mbox{} 

If $Y$ is Schnorr random, then   $\underline \rho(A \leftrightarrow X) = 1/2$ for every computable set $A$.
\end{proof}
\begin{prop} Suppose $A$ is 2-generic. Then $\Delta(A) = 0$. \end{prop}
\begin{proof}
$A$ is neither high nor d.n.c., so $A$ is not in $\+ B(\neq^*)$ as defined in \cite[Section 3.2]{Brendle.Brooke.ea:14}. Hence $A$ does not compute a function in $\mathrm{AED}$,  the mass problem from Subsection~\ref{ss:dualpair} where no computable bound is imposed on the    function. In particular $A$ is does not compute a function in $\mathrm{AED}( \tp{(\tp n)})$, hence $\Delta(A)=0$ by  the second equivalence 	 in  \Cref{thm:mainDelta}. 
\end{proof}
Cor.\ \ref{cor:gamdel} shows that there is no interesting spectrum of values for either $\Gamma$ or $\Delta$, somewhat defeating the orginal purpose of  finely measuring the  complexity of an oracle by comparing it to the computable sets. However, more interesting spectra might be obtained for reducibilities stronger than Turing. Harrison-Trainor~\cite{Harrison-Trainor:17} shows that for many-one reducibility, every value in $[0,1/2]$ is assumed in the case of $\Gamma$. For $\Delta$ this hasn't been studied.

\section{Analog of Theorem~\ref{thm:mainDelta} for cardinal characteristics} \label[section]{s:CC}

As before let  $R \subseteq  X \times Y$ be a relation between spaces $X,Y$; we also assume now that  $\forall x \; \exists y \;
[x R y]$ and $\forall y \; \exists x \; \neg [x R y]$. Let $S = \{  \langle y,x \rangle \in Y \times X \colon \neg xR y\}$.  

\begin{definition} \label[definition]{df: bd} One defines  pairs of  dual cardinal characteristics by 
\begin{eqnarray*} \frd(R) &=& \min\{|G|:G\subseteq Y \land \, \forall x \in X \,
\exists y \in  G   \, xR y\} \\
 \frb(R) = \frd (S) & =&  \min\{|F|:F\subseteq X \land \, \forall y\in Y
\exists x \in F   \, \neg  xR y\}.\end{eqnarray*}
\end{definition}


Note that, in comparison  to \Cref{df:DB}, the defining properties are negated. For a discussion of this, see the beginning of Section 3 of Brendle et al.\ \cite{Brendle.Brooke.ea:14}.

We obtain the characteristics discussed in the introduction as    $\frd(R)$ and $\frb(R)$ for  the two types of relations $R$ introduced in Def.\ ~\ref{df:relations}.  We summarise briefly:

For $  x \in {}^\omega\omega$ and 
  $  y \in \Pi_n \{0, \ldots, h(n)-1\}$,  let  
$x \neq^*_h  y \LR \fa^\infty n \,  [ x(n) \neq y(n)]$.

For $ 0 \le p \le 1/2$,  for $x,y \in {}^\omega 2$, let 
$x \bowtie_p y \LR   \underline \rho (x \lra y)  >p$.


\begin{remark}  It will be convenient for the reader  to  express the characteristics from  \Cref{df: bd} for these relations   in words, with some short notation.

\medskip

\n \fbox{$\frd(\neq^*,h)$} is the least size of a set $G$ of $h$-bounded functions so that for each function $x$ there is a function $y$ in   $G$  such that    $\fa^\infty n [ x(n) \neq y(n)]$.  (Of course it suffices to require this for $h$-bounded $x$. 
The systematic notation is $\frd(\neq^*_h)$.)
\medskip

\n  \fbox{$\frb(\neq^*,h)$} is the least size of a set   $F$ of  functions such that for each $h$-bounded function $y$, there is a function $x$ in $F$  such that $\ex^\infty n \, x(n) = y(n)$.  (We can require that each function in $F$ is $h$-bounded.
 The systematic notation is $\frb(\neq^*_h)$.) 
 \medskip
 
 \n \fbox{$\frd(p)$} is short for  $\frd(\bowtie_p) $, the least size of a set $G$ of bit sequences  such that for each bit sequence $x$ there is a bit sequence  $y$ in  $G$   so that $\uul \rho (x \lra y) >p$. 
 
\medskip

\n   \fbox{$\frb(p)$} is short for  $\frb(\bowtie_p) $, the least size of a set  $F$ of  bit sequences such that for each bit sequence $y$, there is a bit sequence $x$ in $F$  so that $\uul \rho (x \lra y) \le p$.  \end{remark}

Our main goal is to  show that  $\frd(p)= \frd(\neq^*, \tp{(\tp n)})$  and $\frb(p)= \frb(\neq^*, \tp{(\tp n)})$  for each  $p\in (0,1/2)$.   Of course the proofs are similar to the ones in Section~\ref{s:Nies_Delta}, except that the issue of uniformity disappears. See the above-mentioned   Greenberg et al.\ \cite[Thm.\ 6.15]{Greenberg.etal:18} for an exposition  of the results deriving them from a common core.
We begin with some preliminary facts of independent interest.
 The first lemma amplifies bounds  without changing the cardinal characteristics. 

\begin{lemma} \label[lemma]{lem:double}  \bi \item[(i)] Let $h$ be nondecreasing and $g(n) = h(2n)$. 

\n We have $\frd(\neq^*,h) = \frd(\neq^*,g)$ and $\frb(\neq^*,h) = \frb(\neq^*,g)$.

\item [(ii)] For each $a,b >1$ we have $\frd(\neq^*,{\tp{(a^n)})}= \frd(\neq^*,{\tp{(b^n)}})$ and  

\n $\frb(\neq^*,{\tp{(a^n)})}= \frb(\neq^*,{\tp{(b^n)}})$. \ei \end{lemma}

\begin{proof} (i)  Trivially, $h \le g$ implies that   $\frd(\neq^*,h) \ge \frd(\neq^*,g)$ and $\frb(\neq^*,h) \le \frb(\neq^*,g)$. So it suffices to show two inequalities.

\vsps

\n \fbox{$\frd(\neq^*,h) \le\frd(\neq^*,g)$:}  Let $G$ be a witness set for $\frd(\neq^*,g)$. Note that $G$ is also a witness set for $\frd(\neq^*,h(2n+1))$. Let $\hat G= \{p_0 \oplus p_1\colon \, p_0, p_1 \in G \}$, where $(p_0 \oplus p_1)(2m+i) = p_i(m)$ for $i= 0,1$. Each function in $\hat G$ is bounded by $h$.  Since $G$ is infinite, $|\hat G| = |G|$. Clearly $\hat G$ is a witness set for $\frd(\neq^*,h)$.

\vsps

\n \fbox{$\frb(\neq^*,h) \ge \frb(\neq^*,g)$:}  Let $F$ be a witness set for $\frb(\neq^*,h)$.  Let $\hat F$ consist of the functions of the form $ n \to p(2n)$,  or  of the form  $n \to p(2n+1)$,  where $p \in F$. Then $|\hat F| = |F|$, and each function in $\hat F$ is $g$ bounded. 

Clearly,  $\hat F$ is a witness set for $\frb(\neq^*,g)$: if $q$ is $g$-bounded, then  $\hat q$ is $h$ bounded where $\hat q(2n+i) = q(n)$ for $i=0,1$. Let $p\in F$ be such that $\ex^\infty k \, p(k) = \hat q(k)$. Let $i \le 1 $ be such that infinitely many such $k$ have parity~$i$. Then the function $n \to p(2n+i) $, which is in $\hat F$,  is as required.

(ii) is immediate from (i) by iteration. 
\end{proof}
%
%
 
 
\begin{lemma} Let $a \in \omega - \{0\}$. We have $\frd( \neq^*,{\tp{(a^n)}}) \le \frd(1/a)$ and 

\n $\frb( \neq^*,{\tp{(a^n)}}) \ge \frb(1/a)$.\end{lemma}
\begin{proof} As in Section~3,  for $m \ge 2$  let $I_m$  be the $(m-1)$-th consecutive interval of length $a^m$ in $\omega-\{0\}$. First let $G$ be a witness set for $\frd(1/a)$. Let $h(n) = \tp{(a^n)}$.  Recall the operators $K_h$ and $L_h$ from \Cref{df:KhLh}. We show that  $\widehat G = \{ K_h(y) \colon \, y \in G \}$ is a witness set for $\frd( \neq^*,{\tp{(a^n)}})$.
 Otherwise there is a sequence $x\in {}^\omega 2$ such that for each $y \in  \widehat G$ there are infinitely many $m$ with  $x(m) = K_h(y)(m)$. Let $x' =1-x$, that is   $0$s and $1$s are interchanged. Then for each $y \in G$, for  infinitely many $m$, $L_h(x') (i) \neq y(i)$ for each $i \in I_m$.  If we let $n = 1+ \max I_m$, the proportion  of $i < n$ such that $ L_h(x) (i) = y'(i)$ is therefore at most $(a^m-1) / (a^{m+1}-1)$, which converges to $1/a$ as $m \to \infty$. This contradicts the choice of $G$.

Now let $F$ be a witness set for $\frb(\neq^*, h)$. Let $\hat F = \{ 1- L_h(x) \colon \, x \in F \}$. For each  $y  \in \cantor$ there is $x\in F$ such that $\ex^\infty n  \,  K_h(y)(n) = x(n)$. This implies $\uul \rho (y \lra   x') \le 1/a  $ where $  x' =  1- L_h(x) \in \hat F$. Hence $\hat F$ is a witness set for $\frb(1/a)$.
\end{proof}

\begin{theorem} \label[theorem]{th:sht} Fix any  $p\in (0,1/2)$. We have  \bc $\frd(p)= \frd(\neq^*, \tp{(\tp n)})$ and 
$\frb(p)= \frb(\neq^*, \tp{(\tp n)})$. \ec \end{theorem}
\begin{proof} By the two foregoing lemmas we have $  \frd(p)  \ge \frd(\neq^*, \tp{(\tp n)}) $ and $\frb(p)\le  \frb(\neq^*, \tp{(\tp n)})$. It remains to show the converse inequalities:

\n   $\frd(p) \le  \frd(\neq^*, \tp{(\tp n)})$ and $\frb(p)\ge  \frb(\neq^*, \tp{(\tp n)})$.

Recall from Definitions \ref{df: A} and \ref{df: B} that for strings $x,y$ of length $r$,  \[ d(x,y) = \frac 1 r |\{ i \colon x(n) (i)    \neq y(n)(i) \}| \]   If  $h$ is  a function of the form $2^{\hat h}$ with $\hat h \colon \, \omega \to \omega$,    $X=Y= X_h$ denotes  the space of $h$-bounded functions. For   $q \in (0, 1/2)$, 
 we   defined a relation on $X \times Y$ by 
    \[x \neq^*_{\hat h,q} y \LR \fa^\infty n  \, [ d(x(n), y(n)) \ge q].   \]  For ease of notation we continue to denote this relation by $\la \neq^*, \hat h, q \ra$.   
    \begin{claim} \label[claim]{cl:1} For each $c\in \omega$ there is $k \in \omega$ such that
     \begin{eqnarray*}  \frd(q- \frac 2c) &\le & \frd \la \neq^*, \lfloor \tp{n/k} \rfloor, q\ra, \text{ and } \\ 
 \frb(q- \frac 2c) &\ge & \frb \la \neq^*, \lfloor \tp{n/k} \rfloor, q\ra. \end{eqnarray*} \end{claim}
\begin{proof}   \NumberQED{cl:1} As in the proof of \Cref{cl:AR}, let $k $ be large enough so that $\tp{1/k}-1 < \frac 1 {2c}$. Let  $\hat h(n) = \lfloor \tp{n/k} \rfloor$ and  $h = \tp{\hat h}$. Write $H(n) = \sum_{r\le  n} \hat h(r)$.  We refer to the bits with position in $[H(n), H(n+1)) $ as \emph{Block}~$n$. Recall from the proof of \Cref{cl:AR}  that for sufficiently large  $n$ \[  \hat h(n+1) \le \frac 1 c H(n). \]  

For the   inequality involving $\frd$, let $G$ be a witness set for $\frd \la \neq^*, \hat h , q\ra$.  Thus, for each function $x < h$ there is a function $y \in G$ such that for almost all $n$,  $L_h(x), L_h(y)$ disagree on a proportion of $q$ bits of Block $n$. Let $z$ be the complement of $L_h(y)$.   Given $m$, let $n$ be such that $H(n) \le m < H(n+1)$. Since $m- H(n) \le \frac 1 c H(n)$, for large enough $m$, $L_{  h}(x) $ and $z$ agree up to $m$ on a proportion of at least $q - 1.5/c$ bits. So  the set of complements of the $L_h(y)$, $y \in G$, forms a witness set for $\frd(q- 2/c)$ as required.

For the   inequality involving $\frb$, let $F$ be a witness set for  $\frb(q-2/c)$. Thus, for each $y \in \cantor$ there is $x \in F$ such that $\uul \rho(y \lra x) \le q-2/c$. Let $\hat F = \{K_h(1 -x) \colon \, x \in F \}$. We show that $\hat F$ is a witness set for $\frb \la \neq^*, \lfloor \tp{n/k} \rfloor, q\ra$. 

Give a function $y < h$,  let $y' = L_h(y)$.  There is $x \in F$ such that  $\uul {\rho} (y' \lra x) \le q - 2/c$, and hence $\ol  \rho (y' \lra  x' ) \ge 1- q+ 2/c$ where $x' = 1 -x$ is the complement and $\ol  \rho$ denotes the upper density.  Then there are infinitely many $m$ such that the strings $y'\uhr m$ and $x'\uhr m$ agree on a proportion of $> 1- q+1/c$ bits.
Suppose that  $H(n) \le m < H(n+1)$,  then the contribution  of disagreement of   Block  $n$ is at most $1/c$. So 
  there are infinitely many $k$ so that in Block $k$, $y'$ and $x'$ agree on  a proportion of more than $1-q$ bits, and hence disagree on a proportion of fewer than $q$ bits.  
\end{proof}
 
In the following recall \Cref{df: trace}, and in particular that for  $L \in \omega$ and a function~$u$, for any $L$-slalom $s$ and function $y < u$,  
 \[s \not \ni^*_{u,L} y \LR \fa^\infty n [s(n) \not \ni y(n)]. \] 
  We     also write $\la \not \ni^*, u ,L \ra $ for this  relation.

\begin{claim} \label[claim]{cl:2} Given $q < 1/2$, let $L, \epsilon $ be as in \Cref{lem:ListDec}. Fix  a nondecreasing  function $\hat h$, and   let  $u(n)= \tp{\lfloor \epsilon \hat h(n) \rfloor }$. We have 
\[ \frd \la \neq^*,  \hat h , q\ra \le \frd \la \not \ni^*, u,L \ra \text{ and } \frb \la \neq^*,  \hat h , q\ra \ge \frb \la \not \ni^*, u,L \ra. \]
  \end{claim}
\begin{proof} \NumberQED{cl:2} For  the   inequality involving $\frd$, let $G$ be a set of functions bounded by $u$ such that $|G| < \frd \la \neq^*,  \hat h , q\ra$. We show that $G$ is not a witness set  for  the right hand side $\frd \la \not \ni^*, u,L \ra$. 

For each $r$ of the form $\hat h(n)$ choose a set $C= C_r$ as in \Cref{lem:ListDec}. Since $|C_r| = \tp{\lfloor \epsilon r \rfloor}$  we may choose a sequence $\seq {\sss^r_i}_{i<  \tp{\lfloor \epsilon r \rfloor}}$ listing $C_r$ without repetitions. For a  function $y <  u$ let $\wt y$ be the function given by $\wt y(n)= \sss^{\hat h(n)}_{y(n)}$. (Thus,  $\wt y(n)$ is a binary string of length $\hat h(n)$.) Let $\wt G = \{ \wt y  \colon \, y \in G \}$. Then  $|\wt G| \le  |G| < \frd \la \neq^*,  \hat h , q\ra $. So there is a function $x$ with $x(n) \in {}^{\hat h(n)}2$ for each $n$ such that for each $\wt y \in \wt G$ we have  $\ex^\infty n \, [d(x(n), \wt  y(n)) < q]$. Let $s$ be the slalom given by 
\[ s(n) = \{ i \colon \, d (x(n), \sss^{\hat h(n)}_i )< q  \}.\]
Note that by the choice of the $C_r$ according to \Cref{lem:ListDec} and since the listing of $C_r$ has no repetitions,  $s$ is an $L$-slalom. By definition, $\max s(n) < u(n)$. So, for each $y \in G$ we have $\ex^\infty n \, [s(n) \ni y(n)]$. Hence $G$ is not a witness set for $\frd \la \not \ni^*, u,L \ra$.

For  the   inequality involving $\frb$,  suppose $F$ is a witness set  for $\frb \la \neq^*,  \hat h , q\ra$. That is, for each $h= \tp{\hat h}$-bounded function $y$, there is $x\in F$ such that 
\bc $\ex^\infty n \, [ d(x(n), y(n)) <q]$ \ec (as usual we view $x(n), y(n)$ as  binary strings of length $\hat h(n)$).
For $x\in F$ let $s_x$ be the  $L$-slalom such that
 \bc  $s_x(n) = \{i < u(n) \colon \, d(\sss^{\hat h(n)}_i, x(n) )< q\}$. \ec
 Let $\widehat F = \{s_x \colon \, x \in F\}$. Given an $u$-bounded function $y$, let $\wt y(n) =  \sss^{\hat h(n)}_{y(n)}$. There is $x \in F$ such that      $d(x(n), \wt y(n)) < q$ for     infinitely many $n$. This means that $y(n) \in s_x(n)$. Hence $\widehat F$ is a witness set for 
$\frb \la \not \ni^*, u,L \ra$. \end{proof}


We next need an amplification tool   in the context of slaloms. As before, the proof is like  the one of  \Cref{lem:double}(i), so we omit it.

\begin{claim} \label[claim]{lem:double2}     Let $ L \in \omega$, let the function $u$ be nondecreasing and let $w(n) = u(2n)$. We have $\frd(\la \not \ni^*,u, L \ra = \frd\la \not \ni^*,w,L\ra$ and $\frb(\la \not \ni^*,u, L \ra = \frb\la \not \ni^*,w,L\ra$.
 \end{claim}
Iterating the claim,  starting with the function $\hat h(n) = \lfloor \tp{n/k}\rfloor$ with $k$ as in \Cref{cl:1},  we   obtain that $\frd \la \not \ni^*,\tp{\hat h}, L \ra=  \frd \la \not \ni^*,\tp{(L 2^n)}, L \ra$, and similarly

\n $\frb \la \not \ni^*,\tp{\hat h}, L \ra =  \frb \la \not \ni^*,\tp{(L 2^n)}, L \ra$. 
It remains to verify the following.
\begin{claim} \label[claim]{cl:final} \
 \bc $\frd \la \not \ni^*,\tp{(L 2^n)}, L \ra \le \frd(\neq^*, \tp{(\tp n)})$ and 
   $\frb \la \not \ni^*,\tp{(L 2^n)}, L \ra \ge \frb(\neq^*, \tp{(\tp n)})$. \ec \end{claim}
 \begin{proof}
\NumberQED{cl:final}
Given  $n$, recall from Section~3 that we write  a number $k< \tp{(L \tp n)}$ in binary with leading zeros if necessary, and so can view $k$ as a binary string of length $L \tp n$. 

For the inequality involving $\frd$,
let $G$ be a witness set for   $\frd(\neq^*, \tp{(\tp n)})$.  For   functions $y_1, \ldots , y_L$  such that  $y_i(n) < \tp{( \tp n)}$ for each $n$, let $(y_1, \ldots, y_L)$ denote the function $y$ with $y(n) <  \tp{(L \tp n)} $ for each $n$  such that  the $i$-th block of $y(n)$ equals $y_i(n)$ for  each   $i$ with  $1 \le i \le L$.  Let 
\bc $\hat G = \{ (y_1, \ldots, y_n) \colon \, y_1, \ldots, y_L \in G\}$. \ec
Since  $G $ is infinite we have $|\hat G| = |G|$. We check that  $\hat G$ is a witness set for  the left hand side $\frd \la \not \ni^*,\tp{(L 2^n)}\ra$. Given an $L$-slalom $s$ bounded by $\tp{(L \tp n)} $ we may assume that $s(n) $ has exactly $L$ members, and they  are binary  strings of length $L \tp n$. For $i \le L$ let $x_i(n) $ be the $i$-th block of the $i$-th string in $s(n) $,  so that $|x_i(n)|  =  \tp n$. Viewing  the $x_i$ as  functions  bounded by $\tp{(\tp n)}$,  we can choose $y_1, \ldots, y_L \in G$ such that $\fa^\infty n \, x_i(n) \neq y_i(n)$. Let $y = (y_1, \ldots, y_n) \in \hat G$. Then $\fa^\infty n \, [s(n) \not \ni  y(n)]$,  as required. 

For the inequality involving $\frb$,  let $F$ be a witness set for $\frb \la \not \ni^*,\tp{(L 2^n)}$. That is,  $F$ is a set of $L$-slaloms $s$ such that for each function $y$ with $y(n) < 2^{(L 2^n)}$, there is $s\in F$ such that $s(n) \ni y(n)$ for infinitely many $n$.  

Let $\hat F$ be the set of functions $s_i$, for $s\in F$ and  $i<L$, such that $s_i(n) $ is the $i$-th block of the $i$-th element of $s(n)$ (as before we may assume that each string in $s(n)$ has length $L2^n$). Now let $y$ be a given function bounded by $2^{(2^n)}$. Let $y'$ be the function bounded by $2^{(L2^n)}$ such that for each $n$, each block of $y'(n)$ equals $y(n)$. There is $s\in F$ such that $s(n) \ni y'(n)$ for infinitely many $n$. There is $i<L$ such that  $y'(n)$ is the $i$-th string in $s(n)$ for infinitely many of these $n$, and hence  $y(n) = s_i(n)$. Thus  $\hat F$ is a witness set for $\frb(\neq^*, \tp{(\tp n)})$.
\end{proof}

We can now summarise the argument  that  $\frd(p) \le  \frd(\neq^*, \tp{(\tp n)})$ for $p < 1/2$. Pick $c$ large enough such that $q= p+2/c < 1/2$.  
By \Cref{cl:1} there is $k$ such that   $\frd(p) \le \frd \la \neq^*, \lfloor \tp{n/k} \rfloor, q\ra$.   
 By \Cref{cl:2} there are $L$, $\epsilon$ such that where $\hat h(n) =  \lfloor \tp{n/k} \rfloor$, we have    $\frd \la \neq^*,  \hat h , q\ra \le \frd \la \not \ni^*, u,L \ra$,     
where $u(n)= \tp{\lfloor \epsilon \hat h(n) \rfloor } $.
Applying \Cref{lem:double2}  sufficiently many times we  have   $\frd \la \not \ni^*, u,L \ra \le \frd \la \not \ni^*, \tp{(L 2^n)} ,L \ra$.    
Finally, $\frd \la \not \ni^*,\tp{(L 2^n)}, L \ra \le \frd(\neq^*, \tp{(\tp n)})$ by \Cref{cl:final}.

 The argument for $\frb(p)\ge  \frb(\neq^*, \tp{(\tp n)})$, $p < 1/2$,  is  dual to the above. 
\end{proof}

\section{A proper hierarchy of problems  $\mathrm{IOE}(h)$ in the weak degrees}\label[section]{s:fcuk}Recall that by $\mathrm{IOE}(h)$ we denote  the mass problem  of functions  $f$  such that $\ex^\infty n \, [f(n) = r(n) ]$ for each computable function $r< h$.
In this section we study   how the  Muchnik degree of    $\mathrm{IOE}(h)$ depends on the function~$h$. In \cite{Monin.Nies:15} the authors obtained the following two results:
\begin{thm}[\cite{Monin.Nies:15}, Thm.\ IV.1]
Let $c \geq 2$ be any integer, which we view as a constant function. 
\bc $\mathrm{IOE}(  2) \equiv_S \mathrm{IOE}( c) \equiv_S  \{X \colon \, X \text{ is not computable}\}$ \ec
\end{thm}
The difficult part of the theorem is to show that $\mathrm{IOE}(  2) \geq_S \mathrm{IOE}(  c)$ for $c > 2$. This can be done using error-correcting codes.

\begin{thm}[\cite{Monin.Nies:15}, Section 4]
For any pair of order functions $F < G$ such that 
 $\sum_n 1/F(n) = \infty$ and $\sum_n 1/G(n) < \infty$, we have $\mathrm{IOE}(F) <_W \mathrm{IOE}(G)$.  
\end{thm}

We now show that given any order  function $F$, one can find a function $G > F$ such that: 
\bc $\mathrm{IOE}(F) <_W \mathrm{IOE}(G)$ \ec

Given an order function $F$, we let $w_F(n)$ be the number of possible combinations of $n$ first values for functions $f \leq F$, that is, \bc $w_F(n) = \Pi_{0 \leq i \leq n} F(i)$. \ec
To improve  the  readability  of expressions with iterated exponentiation, we will mostly write    $  \exp(x)$  for  $2^x$.

\begin{theorem}[with Joseph Miller] \label[theorem]{th-ioe}
Let $F \in {}^\omega \omega$ be an order function. Let $G \in {}^\omega \omega$ be an order function with $G(n) \geq 2$ for every $n$ and such that:
$$\forall k\ \forall^\infty n\ \exp(w_{F}(\exp(n \cdot k)) )< G(n)$$
There exists a function $f \in \mathrm{IOE}(F)$ and such that $g \notin \mathrm{IOE}(G)$ for every $g \leq_T f$.
\end{theorem}
\n For instance, if $F(n) = n$ we can let  $G(n) = \exp \exp \exp (n^2)$. 

 The rest of the section is dedicated to the proof of \Cref{th-ioe}. Let us first introduce some terminology.
\begin{definition}
By a \emph{tree} we mean a set of strings closed under prefixes. Let $H : \omega \rightarrow \omega$. We denote by ${}^{<\omega} H$ the tree consisting of the strings $\sss$ such that $\sss(i) < H(i)$ for each $i < \sssl$.

 Let $T \subseteq {}^{<\omega} H$ be a tree. We say that $T$ is \emph{$H$-full-branching} if for every $f < H$ we have $f \in [T]$. For a string $\sigma \in {}^{<\omega} \omega$ and $n > |\sigma|$, we say that $T$ is $H \rest n$-\emph{full-branching above $\sigma$} if for every $f < H$ with $\sigma \prec f$ we have $f \rest n\ \in T$.
\end{definition}

Given a node $\sigma$ of length $m$ and a $H \rest {m+n}$-full-branching tree $T$ above $\sigma$, we sometimes say that $n$ is the \emph{height} of the full-branching part of $T$. We begin with one of those lemma whose statement is more complicated than the proof.

\begin{lemma} \label[lemma]{lemma:complicated}
{\rm Let $H : \omega \rightarrow \omega$. Let $\sigma \in {}^{<\omega} H$ with $|\sigma| = m$. Let $n \in \omega$ and let $T \subseteq {}^{<\omega} H$ be a finite $H \rest {m + 2n}$-full-branching tree above $\s$. Let $\s_0, \dots, \s_{k}$ be all the leaves of $T$. Consider a partition $C_1,C_2$ of these leaves. Then one of the following holds:
\begin{enumerate}
\item[(i)] If we keep only the nodes compatible with some element of $C_1$ and discard the rest, the remaining tree is $H \rest {m + n}$-full-branching above~$\s$.
\item[(ii)]  If we keep only the nodes compatible with some element of $C_2$ and discard the rest, there exists a node $\tau \succ \sigma$ of length $m+n$ such that the remaining tree is $H \rest {m + 2n}$-full-branching above $\tau$.
\end{enumerate} 
In particular, in both cases, the full-branching part of the remaining tree has height $n$.}
\end{lemma}
\begin{proof}
Suppose (i) fails. Then there is a string $\tau \succ \s$, $\tau \in T$  of length $m+n$ such that all the extensions in $T$ of length $m+2n$ of $\tau$ are leaves of $T$ which are not in $C_1$. Then  these leaves are in $C_2$. So (ii) holds.
\end{proof}

Given any functional $\Phi$, we will be  able to compute an infinite tree $T \subseteq {}^{<\omega} F$ such that:
\begin{enumerate}
\item[(1)] For every path $X \in [T]$ we have that $\Phi(X) \notin \mathrm{IOE}(G)$.
\item[(2)] For every path $X \in [T]$, there are infinitely many $m$ such that $T$ is $F \rest {m+1}$-full-branching above $X \rest m$.
\item[(3)] $T$ has no dead ends.
\end{enumerate}
Note that (3) ensures that the tree is computable in a strong sense : if a node $\s $ is in $ T$, then  there exists an infinite path $X \in [T]$ with $X \succ \s$. By combining (2) with (3) we actually know that the set of infinite paths extending $\s$ is perfect. While (1) ensures that no path of $T$ computes an element of $\mathrm{IOE}(G)$ via $\Phi$, (2) ensures that the tree $T$ still contains an element of $\mathrm{IOE}(F)$. Also,  starting from the tree ${}^{<\omega} F$, one can  compute a sub-tree $T$ which satisfies (1) and (2) using \Cref{lemma:complicated}. 

In order to help the reader understand the full proof, we sketch here a construction to obtain, under the assumption that $G$ grows sufficiently faster than $F$, a computable tree $T$ that satisfies    (1) -- (3) given some  functional $\Phi$. Of course this allows us to ``defeat" only one functional $\Phi$. To defeat more than one functional $\Phi$ we would need not only to obtain (2), but to obtain a computable tree for which we have infinitely many large full-branching blocks. In this case we can repeat the construction in the tree we end up with,  so as to  defeat yet another  functional. This  will be achieved by the upcoming \Cref{lemma:all},  elaborating on the ideas already present in the   construction we discuss now.

\begin{proof}[Sketch of a construction to obtain (1), (2) and (3)]

We work here under the assumptions of \Cref{th-ioe}. Note however that in the simpler case of defeating only one functional, the assumption on how fast $G$ grows compared to $F$ can be relaxed somewhat: we merely need that
$$\forall k\ \forall^\infty n\ \exp(k \cdot F(k + \exp(n))) < G(n).$$ 
In the following all strings will be chosen from the $F$-full-branching tree. We can suppose without loss of generality that given any $\s$ and any $n$ there exists an extension $\tau$ of $\s$ such that $\Phi(\tau, n)$ is defined. Otherwise there is a string $\s$ and some $n$ such that $\Phi(X, n)$ is undefined for every path $X$ extending $\s$ and the desired tree $T$ is given by all the nodes compatible with~$\s$. 

The construction  inductively  defines     finite trees $T_0 \subset T_1 \subset \dots$ together with integers $n_0 < n_1 < \dots$ such that :
\begin{enumerate}
\item[(a)] For every $k$, every leaf of $T_k$ has a full-branching extensions in $T_{k+1}$.
\item[(b)] For every $k$, every leaf $\rho \in T_k$ and every $t < n_k$, every value $\Phi(\rho, t)$ is defined.
\item[(c)] For every $k$, every $t \leq n_k$ one value smaller than G(t) is different from $\Phi(\rho, t)$ for every leaf $\rho \in T_k$.
\item[(d)] For every $k$ we have $\exp(c) < G(n_k)$ where $c$ is the number of leaves in $T_k$
\end{enumerate}

Note that unlike (a) (b) and (c), (d) does not achieve by itself anything we want, but it will be necessary at each step to continue the induction, in particular in order to show (c). 

To begin the inductive definitions,  let $n_0$ be least such that $$\exp(F(\exp(n_0))) < G(n_0).$$  Consider the $F \rest{\exp(n_0)}$-full-branching tree above the empty string. Let $\s_0, \dots, \s_{c}$ be an enumeration of the leaves of this $F \rest{\exp(n_0)}$-full-branching tree. For each $i \leq c$, we look for an extension $\tau_i$ of $\s_i$ such that $\Phi(\tau_i, t)$ is defined for every $t \leq n_0$. We can assume without loss of generality that every node $\tau_i$ has the same length $m_0$ (presumably much larger than $n_0$). We now partition the set of nodes $\tau_i$ into those such that $\Phi(\tau_i, 0) = 0$ and those such that $\Phi(\tau_i, 0) \neq 0$. By \Cref{lemma:complicated}, we can either remove all nodes of length $m_0$ forcing $\Phi(0) = 0$, or all nodes of length $m_0$ forcing another value (and everything compatible with these nodes), in such a way that we have a node $\s$ above which the tree consisting of the nodes we keep is $F \rest {|\sss| + \exp(n_0-1)}$-full branching. Note that $\sigma$ can be either the root of the tree or a string  of length $\exp(n_0-1)$.

We inductively continue  the previous operation for each of the $n_0$ first values of $\Phi$. At the end, we have a node $\s$ above which there is a $F \rest {|\sigma| + 1}$-full-branching tree, and such that given any $t \leq n_0$, the remaining nodes $\tau_i$ of length $m_0$ are altogether such that $\Phi(\tau_i, t)=0$ or such that $\Phi(\tau_i, t) \neq 0$. Let $T_0$ be the tree consisting of the remaining nodes $\tau_i$ and everything below them. For every $t \leq n_0$, in the first case we define $g(t) = 1$ and in the second $g(t) = 0$. Note that as $\exp(F(\exp(n_0))) < G(n_0)$, then also we must have $\exp(c) < G(n_0)$ where $c$ is the number of nodes of length $m_0$ in $T_0$.

Suppose now by induction that we have a finite tree $T_k$ with leaves $\tau_0, \dots, \tau_c \in T_k$ each of length $m_k$, and a value $n_k$ such that (a), (b), (c) and (d) are verified. In particular we have $\exp(c) < G(n_k)$. Let $n_{k+1} > n_k$ be the smallest such that 
$$\exp(c \cdot F(m_k + \exp(n_{k+1}))) < G(n_{k+1})$$
Let us show that for any $a$ with $0 \leq a \leq c$, we can computably find a finite tree $T^a$ whose nodes are all compatible with $\tau_a$ and such that:

\begin{itemize}
\item $T^a$ is $F \rest {|\s| + 1}$-full branching above some $\s \succ \tau_a$.
\item Each leaf $\rho$ of $T^a$ is such that $\Phi(\rho, t)$ is defined for $n_{k} < t \leq n_{k+1}$.
\item For every $n_k < t \leq n_{k+1}$, there is at least one value smaller than $G(t)$ which is different from every value $\Phi(\rho, t)$ for leaves $\rho$ of $T^a$.
\end{itemize}


For any $a \leq c$ we do the following: consider the finite $F \rest {|\tau_a| + \exp(n_{k+1})}$-full branching tree above $\tau_a$. Let $\s_0, \dots, \s_{k}$ be an enumeration of the leaves of this finite tree. For each of these nodes $\s_i$, look for an extension $\tau_i'$ such that $\Phi(\tau_i', t)$ is defined for every $n_k < t \leq n_{n+1}$. Let ${T^a}'$ be the finite tree consisting of these extensions $\tau_i'$ and everything below them.

We now partition the set of leaves of ${T^a}'$ into two sets $C_1$ and $C_2$ such that the leaves $\rho$ in $C_1$ are these for which the $a$-th bit of $\Phi(\rho, n_{k}+1)$ is $0$ and the leaves in $C_2$ are these for which the $a$-th bit of $\Phi(\rho, n_{k}+1)$ is $1$. By the \Cref{lemma:complicated}, we can either remove all nodes of $C_1$ or all nodes of $C_2$ (and everything compatible with these nodes), in such a way that we have a node $\s \in {T^a}'$ such that the tree consisting of the nodes we keep, is $F\rest {|\s| + \exp(n_{k+1}-1)}$-full branching above~$\s$.

We inductively continue  the previous operation for each of the next values of $\Phi$ up to $n_{k+1}$. At the end, we have a node $\s \in {T^a}'$ above which there is a $F \rest {|\s| + 1}$-full-branching tree as follows: for each $n_k < t \leq n_{k+1}$, for all the remaining leaves $\tau'$ of our $F \rest {|\s| + 1}$-full-branching tree, the $a$-th bit of $\Phi(\tau', t)$ is the same. We define the tree $T^a$ to be this set of remaining leaves $\tau'$ and everything below them.

Once every tree $T^a$ has been defined, we define each value of $g(t)$ for $n_k < t \leq n_{k+1}$, as follows: If the leaves $\rho$ of $T^a$ are such that the $a$-th bit of $\Phi(\rho, t)$ equals $0$, then the $a$-th bit of $g(n)$ is defined to be $1$, and vice-versa. Recall that we have $\exp(c) < G(n_k)$. In particular any number coded on at most $c$ bits is smaller than $G(t)$ for any $n_k < t \leq n_{k+1}$. It follows that $g(t) < G(t)$ for any $n_k < t \leq n_{k+1}$. Also we necessarily have that $g(t)$ is different from every possible value $\Phi(\rho, t)$ for every leaf $\rho \in \bigcup_{a \leq c} T^a$. Let $T_{k+1} = \bigcup_{a \leq c} T^a$. Note that by the choice of $n_{k+1}$ we have that $\exp(d) < G(n_{k+1})$ where $d$ is the number of leaves in $T_{k+1}$.

By continuing the induction, we define a computable subtree $T = \bigcup_{k} T_k$ of the $F$-full-branching tree as well as a computable function $g < G$, such that along any path of $T$, infinitely many nodes are full-branching, and such that for any $f \in [T]$ we have that $\Phi(f, n) \neq g(n)$ for any $n$.
\end{proof}

  Suppose now that we want to defeat every functional. Let  $\Phi_0, \Phi_1, \Phi_2, \dots$ be a list of all functionals. The previous proof gives us a tree $T_0$ which defeats $\Phi_0$. To defeat $\Phi_1$, we have to perform a similar construction, but starting now from the computable tree $T_0$ in place of the $F$-full-branching tree ${}^{<\omega} F$. In this way  we obtain a computable tree $T_1 \subseteq T_0$ which defeats both $\Phi_0$ and $\Phi_1$. The main problem is that to use \Cref{lemma:complicated} we need to work in a tree that has large full-branching blocks (which is the case of ${}^{<\omega} F$). Also if $T_0$ itself does not have large full-branching blocks, it is not necessarily possible to defeat $\Phi_1$ starting from $T_0$ in place of ${}^{<\omega} F$. To overcome this problem, it is not sufficient  to merely ensure (2) for $T_0$: we actually  need to ensure that for every path $X \in T_0$, there are infinitely many $m$ such that $T_0$ is $F \rest {m+n_m}$-full-branching above $X \rest m$ for $n_m$ sufficiently large. This leads to the following definition:

\begin{definition}
Let $F,G \in {}^{\omega} \omega$ be order functions. Let $T \subseteq {}^{<\omega} F$ be a finite tree. Let $n_1 < n_2 < \dots < n_k$. We say that $T$ is \emph{$G$-fat for $(n_1, n_2, \dots, n_k)$} if for every leaf $\sigma \in T$, there exists $m_1 < m_2 < \dots < m_k < |\s|$ such that for every $1 \leq t \leq k$:
\begin{enumerate}
\item The tree $T$ is $F \rest {m_t + \exp(n_t \cdot t)}$-full-branching above $\sigma \rest {m_t}$.
\item $\exp(w_{F}(m_t + \exp(n_t \cdot t))) < G(n_t)$.
\end{enumerate}
We say that $T\subseteq {}^{<\omega} F$ is \emph{infinitely often $G$-fat} if there exists an infinite sequence $n_1 < n_2 < \dots$ such that for every $k$, there exists $m$ such that $T$ restricted to its node of length $m$, is $G$-fat for $(n_1, \dots, n_k)$.
\end{definition}

The following lemma is the heart of the proof. It says that for  any computable infinitely often $G$-fat tree $T$ and any functional $\Phi$, there is  a computable  infinitely often $G$-fat tree $T' \subseteq T$ such that no path of $T'$ computes an element of $\mathrm{IOE}(G)$ via $\Phi$.

\begin{lemma} \label[lemma]{lemma:all}
Let $F \in {}^{\omega} \omega$ be an order function. Let $G \in {}^{\omega} \omega$ be an order function such that $G(n) \geq 2$ for every $n$. Let $T \subseteq {}^{<\omega} F$ be a computable infinitely often $G$-fat tree with no dead ends. Let $\Phi$ be a functional. There exists a computable infinitely often $G$-fat tree $T' \subseteq T$ with no dead ends, and a computable function $g < G$ such that for every path $X \in [T']$ for which $\Phi(X)$ is total, we have $\Phi(X, n) \neq g(n)$ for every $n$.
\end{lemma}

Before giving the proof of the Lemma, we show how to use it in order to obtain the proof of \Cref{th-ioe}, using simple forcing machinery.

\begin{proof}[Proof of \Cref{th-ioe}]
Let $F \in {}^\omega \omega$ be an order function. Let $G \in {}^\omega \omega$ be an order function with $G(n) \geq 2$ for every $n$ and such that:
$$\forall k\ \forall^\infty n\ \exp(w_{F}(\exp(n \cdot k)) < G(n)$$
Let us show that there exists a function $f \in \mathrm{IOE}(F)$ and such that $g \notin \mathrm{IOE}(G)$ for every $g \leq_T f$. The proof is done by forcing, using \Cref{lemma:all}. We first need to argue that under the above hypothesis, the tree $T = {}^{<\omega} F$ is infinitely often $G$-fat. In what follows, the notation $T \rest n$ refers to the finite tree consisting of the nodes of $T$ of length smaller than or equal to $n$. Let $n_1$ be the smallest such that $\exp(w_{F}(\exp(n_1))) < G(n_1)$. The tree $T \rest {\exp(n_1)}$ is $F \rest {\exp(n_1)}$-full-branching above the empty string and in particular the tree $T \rest {\exp(n_1)}$ is $G$-fat for $(n_1)$. Suppose now that we have defined $n_1 < \dots < n_k$ such that $T \rest {\exp(n_k)}$ is $G$-fat for $(n_1, \dots, n_k)$. Let $n_{k+1}$ be the smallest such that 
$$
\begin{array}{rccl}
&\exp(n_k) + \exp(n_{k+1} \cdot (k+1))&<&\exp(n_{k+1} \cdot (k+2))\\
\text{ and }&\exp(w_{F}(\exp(n_{k+1} \cdot (k+2))))&<&G(n_{k+1})
\end{array}
$$
Then in particular we have 
$$\exp(w_{F}(\exp(n_k) + \exp(n_{k+1} \cdot (k+1)))) < G(n_{k+1})$$
It follows that the tree $T \rest {\exp(n_{k+1})}$ is $G$-fat for $(n_1, \dots, n_k, n_{k+1})$. Therefore the tree $T$ is infinitely often $G$-fat for the infinite sequence $\{n_k\}_{1 \leq k < \omega}$.

So we start the forcing with the tree $T = {}^{<\omega} F$. Let $\P$ be the set of forcing conditions consisting of all the computable infinitely often $G$-fat subtrees of $T$ with no dead ends. For two forcing conditions $P_1,P_2 \in \P$, the partial order $P_2 \preceq P_1$ is defined by $P_2 \subseteq P_1$. Let $\Phi$ be a functional. By \Cref{lemma:all}, the set of infinitely often $G$-fat trees $P \in \P$ such that for every path $X$ of $[P]$ we have $\Phi(X) \notin \mathrm{IOE}(G)$, is dense in $\P$.

We simply have to argue that for any computable function $f < F$, the set of infinitely often $G$-fat trees $P \in \P$ such that every path $X$ of $[P]$ equals at least once to $f$, is dense in $\P$. It is clear, because given a tree $P \in \P$, consider any node $\tau \in P$ of length $m$ such that $P$ is $F\rest {m+1}$-full-branching above $\tau$. Let $\tau'$ equals $\tau \cat f(m+1)$. Note that $\tau' \in P$. Now let $P'$ to be the nodes of $P$ which are compatible with $\tau'$. It is clear that $P' \in \P$ and that $P' \preceq P$. Thus the set of infinitely often $G$-fat trees $P \in \P$ such that every path $X$ of $P$ equals at least once to $f$, is dense in $\P$.

Consider now any sufficiently generic set of conditions $\{P_n\}_{n \in \omega}$ with $P_1 \succ P_2 \succ \dots$. We have that $\bigcap_n P_n$ contains at least one infinite path $X$. Also this path necessarily equals at least once every computable function bounded by $F$, and thus equals infinitely often every computable function bounded by $F$. It follows that $X \in \mathrm{IOE}(F)$. Furthermore for any function $\Phi$ we have that $\Phi(X) \notin \mathrm{IOE}(G)$. This shows the theorem.
\end{proof}

%

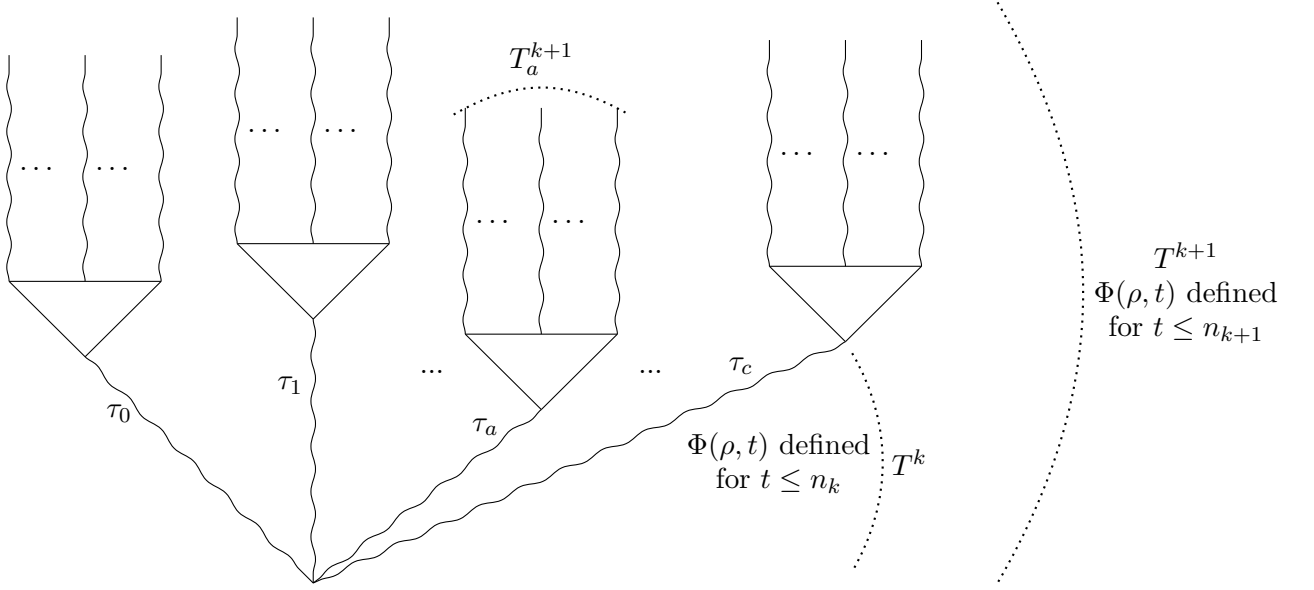
\begin{figure}
\begin{center}
\hspace*{0in}
\resizebox{\textwidth}{!}{\begin{tikzpicture}[every text node part/.style={align=center}]
\tikzset{every tree node/.style={align=center}}

\coordinate[] (t10);
\coordinate[above left=1cm and 1cm of t10] (t11);
\coordinate[right=1cm of t11] (t1m);
\coordinate[right=2cm of t11] (t12);

\coordinate[above=3cm of t11] (h11);
\coordinate[right=1cm of h11] (h1m);
\coordinate[right=2cm of h11] (h12);

\coordinate[above right=.5cm and 3cm of t10] (t20);
\coordinate[above left=1cm and 1cm of t20] (t21);
\coordinate[right=1cm of t21] (t2m);
\coordinate[right=2cm of t21] (t22);

\coordinate[above=3cm of t21] (h21);
\coordinate[right=1cm of h21] (h2m);
\coordinate[right=2cm of h21] (h22);

\coordinate[below right=1.2cm and 3cm  of t20] (t30);
\coordinate[above left=1cm and 1cm of t30] (t31);
\coordinate[right=1cm of t31] (t3m);
\coordinate[right=2cm of t31] (t32);

\coordinate[above=3cm of t31] (h31);
\coordinate[right=1cm of h31] (h3m);
\coordinate[right=2cm of h31] (h32);

\coordinate[above right=.9cm and 4cm of t30] (t40);
\coordinate[above left=1cm and 1cm of t40] (t41);
\coordinate[right=1cm of t41] (t4m);
\coordinate[right=2cm of t41] (t42);

\coordinate[above=3cm of t41] (h41);
\coordinate[right=1cm of h41] (h4m);
\coordinate[right=2cm of h41] (h42);

\coordinate[below right=3cm and 3cm of t10] (root);
\coordinate[right=7cm of root] (inv1);
\coordinate[right=2cm of inv1] (inv2);
\coordinate[above=7.7cm of inv2] (inv3);

\draw[-] (t10) to (t11);
\draw[-] (t11) to (t12);
\draw[-] (t12) to (t10);

\draw[-] (t20) to (t21);
\draw[-] (t21) to (t22);
\draw[-] (t22) to (t20);

\draw[-] (t30) to node[left] {...\ \ \ \ \ } (t31);
\draw[-] (t31) to (t32);
\draw[-] (t32) to node[right] {\ \ \ \ \ ...} (t30);

\draw[-] (t40) to (t41);
\draw[-] (t41) to (t42);
\draw[-] (t42) to (t40);

\draw[-,decorate,decoration={snake,amplitude=.3mm,segment length=6mm,post length=1mm}] (t10) to  node[near start, left] {$\tau_0$} (root);
\draw[-,decorate,decoration={snake,amplitude=.3mm,segment length=6mm,post length=1mm}] (t20) to node[near start, left] {$\tau_1$} (root);
\draw[-,decorate,decoration={snake,amplitude=.3mm,segment length=6mm,post length=1mm}] (t30) to node[pos=.1, left] {$\tau_a$\ } (root);
\draw[-,decorate,decoration={snake,amplitude=.3mm,segment length=6mm,post length=1mm}] (t40) to node[pos=.1, left] {$\tau_c$\ \ \ } (root);

\draw[-,decorate,decoration={snake,amplitude=.3mm,segment length=6mm,post length=1mm}] (t11) to node[right] {$\dots$} (h11);
\draw[-,decorate,decoration={snake,amplitude=.3mm,segment length=6mm,post length=1mm}] (t1m) to node[right] {$\dots$} (h1m);
\draw[-,decorate,decoration={snake,amplitude=.3mm,segment length=6mm,post length=1mm}] (t12) to (h12);

\draw[-,decorate,decoration={snake,amplitude=.3mm,segment length=6mm,post length=1mm}] (t21) to node[right] {$\dots$} (h21);
\draw[-,decorate,decoration={snake,amplitude=.3mm,segment length=6mm,post length=1mm}] (t2m) to node[right] {$\dots$} (h2m);
\draw[-,decorate,decoration={snake,amplitude=.3mm,segment length=6mm,post length=1mm}] (t22) to (h22);

\draw[-,decorate,decoration={snake,amplitude=.3mm,segment length=6mm,post length=1mm}] (t31) to node[right] {$\dots$} (h31);
\draw[-,decorate,decoration={snake,amplitude=.3mm,segment length=6mm,post length=1mm}] (t3m) to node[right] {$\dots$} (h3m);
\draw[-,decorate,decoration={snake,amplitude=.3mm,segment length=6mm,post length=1mm}] (t32) to (h32);

\draw[-,decorate,decoration={snake,amplitude=.3mm,segment length=6mm,post length=1mm}] (t41) to node[right] {$\dots$} (h41);
\draw[-,decorate,decoration={snake,amplitude=.3mm,segment length=6mm,post length=1mm}] (t4m) to node[right] {$\dots$} (h4m);
\draw[-,decorate,decoration={snake,amplitude=.3mm,segment length=6mm,post length=1mm}] (t42) to (h42);

\draw[-, dotted, thick, shorten <=-5pt,shorten >=-5pt] (h31) to[bend left] node[above] {$T_a^{k+1}$} (h32);

\draw[-, dotted, thick, shorten <=5pt,shorten >=5pt] (t40) to[bend left] node[right] {$T^{k}$} node[left] {$\Phi(\rho, t)$ defined\\for $t \leq n_{k}$} (inv1);
\draw[-, dotted, thick] (inv3) to[bend left] node[right] {$T^{k+1}$\\ $\Phi(\rho, t)$ defined\\for $t \leq n_{k+1}$} (inv2);
\end{tikzpicture}}
\end{center}
\caption{Construction of $T_{k+1}$ from $T_k$ in \Cref{lemma:all}. We have $2^c~<~G(n_k)$ and for every $n_k < t \leq n_{k+1}$, the $a$-th bit of $\Phi(\rho, t)$ is the same for every $\rho \in T_a^{k+1}$}.
\label{fig:1}
\end{figure}

\begin{proof}[Proof of \Cref{lemma:all}]
 Figure \ref{fig:1}    illustrates a part of the proof. Suppose first that there exists a node $\sigma \in T$ such that for every $X \succ \sigma$ with $X \in [T]$, we have that $\Phi(X)$ is partial. Then we define the computable tree $T'$ to be the nodes of $T$ compatible with $\sigma$. It is clear that $T'$ is infinitely often $G$-fat. Also as $\Phi(X)$ is partial for every $X \in [T']$ this case of the  lemma is verified.

So we can now suppose without loss of generality that for every node $\sigma \in T$ and every $n$, there exists an extension $\tau \succeq \sigma$ such that $\Phi(\tau, n)$ is defined. 
From $T$ we want to find $T' \subset T$ as in the lemma. This is done step-by-step. At each step $k$ we find values $n_1 < \dots < n_k$ and a finite tree $T_k \supseteq T_{k-1}$ such that $T_k$ is $G$-fat for $(n_1 < n_2 < \dots < n_k)$ and such that for leaves $\rho$ of $T_k$, the values $\Phi(\rho, e)$ are all different from something smaller than $G(e)$ for every $e \leq n_k$. However, we do not show right away that the values $\Phi(\rho, e)$ are all different from something smaller than $G(e)$. We first show that we can make large group of leaves which all agree on a specific bit. The fact that we can use that to have the values $\Phi(\rho, e)$ all different from something smaller than $G(e)$ will be made clear later. Here is a claim which says how one step is done : building the tree $T_k$ from the tree $T_{k-1}$.

\begin{claim} 
  Let $T \subseteq {}^{<\omega} F$ be a computable infinitely often $G$-fat tree.
Let $n_1 < n_2 < \dots < n_{k-1}$. Suppose that a finite tree $T_{k-1} \subseteq T$ is $G$-fat for $(n_1 < n_2 < \dots < n_{k-1})$. Let $\sigma_0, \dots, \sigma_c$ be the leaves of $T_{k-1}$. Then there exists $n_{k}$ such that above each node $\sigma_a$ for $0 \leq a \leq c$, we can find an extension $\tau_a \succeq \sigma_a$ of length $m_a$ and a finite tree $T^a \subseteq T$ whose nodes are all comparable with $\tau_a$ and such that: 

\begin{enumerate}
\item[(1)] $T^a$ is $F \rest {m_a + \exp(n_k \cdot k)}$-full-branching above $\tau_a$.
\item[(2)] For every $e$ with $n_{k-1} < e \leq n_k$ and every leaf $\rho \in T^a$, the value $\Phi(\rho, e)$ is defined.
\item[(3)] For every $e$ with $n_{k-1} < e \leq n_k$, there exists $i \in \{0,1\}$ such that for every leaf $\rho \in T^a$, the $a$-th bit of $\Phi(\rho, e)$ equals $i$.
\item[(4)] $\exp(w_{F}(m_a+\exp(n_k \cdot k))) < G(n_k)$. 
\end{enumerate}
In particular, letting $T_{k} = \bigcup_{a < c} T^a$, we have that $T_{k} \subseteq T$ is $G$-fat for $(n_1 < n_2 < \dots < n_k)$.\end{claim}

We first show how to use this claim  in order to build the tree $T'$ and the computable function $g$ of the lemma. At step $1$ we apply the claim starting from the empty tree, with the empty string as the only leaf. The claim gives us some $n_1 > 0$ and a finite subtree $T_1 \subseteq T$ which is $G$-fat for $(n_1)$ and such that for every $e \leq n_1$, the first bit of $\Phi(e, \rho)$ is the same for every leaf $\rho$ of $T_1$. We define in the mean time the computable function $g(e)$ for $e \leq n_1$ so that its first bit is different from the one forced on leaves of $T_1$. Note that $g(e) \in \{0,1\}$ and that as $G(e) \geq 2$ we necessarily have $g(e) < G(e)$ for $e \leq n_1$. We now deal with a crucial point for the rest of the induction, corresponding to the point (d), in the proof that defeats only one functional. As $T_1$ is $G$-fat for $(n_1)$, there exists a node $\tau_1 \in T_1$ of length $m_1$ such that $T_1$ is $F \rest {m_1 + \exp(n_1)}$-full-branching above $\tau_1$ and such that $\exp(w_{F}(m_1+(\exp(n_1)))) < G(n_1)$ (using (4) of he claim). Let $c$ be the number of leaves of $T_1$. Note that $w_{F}(m_1+(\exp(n_1))$ is the number of nodes in the $F\rest {m_1+\exp(n_1)}$-full-branching tree above the empty string. As $c$ is the number of nodes in the $F\rest {m_1+\exp(n_1)}$-full-branching tree above $\tau_1$, it follows that $c \leq w_{F}(m_1+\exp(n_1))$ and then that $\exp(c) < G(n_1)$. Just as in the proof that defeats only one functional, this will allow us to continue the induction and in particular to have values smaller than $G(n_1 + t)$ for which we can continue to define $g$.

Suppose now by induction that at step $k$ we have a sequence $n_1 < \dots < n_k$ and a finite tree $T_k \subseteq T$ which is $G$-fat for $(n_1, \dots, n_k)$. Let $\sigma_0, \dots, \sigma_c$ be the leaves of $T_k$ and suppose also that $c$ is such that $\exp(c) < G(n_k)$. Let us define $n_{k+1} > n_k$ and a finite tree $T_{k+1} \subseteq T$, $G$-fat for $(n_1, \dots, n_k, n_{k+1})$, and which extends $T_k$, together with values $g(e)$ for $n_k < e \leq n_{k+1}$ such that $g(e) < G(e)$ and such that $g(e)$ is different from $\Phi(e, \rho)$ for every leaf $\rho$ of $T_{k+1}$. Using the above claim, we find $n_{k+1} > n_k$ and above each node $\sigma_a$ for $a \leq c$ we find an extension $\tau_a \succeq \sigma_a$ of length $m_a$ and a finite tree $T_{k+1}^a \subseteq T$ such that $T_{k+1}^a$ is $F \rest {m_a + \exp(n_{k+1} \cdot (k+1))}$-full-branching above $\tau_a$. Also for every $e$ with $n_{k} < e \leq n_{k+1}$, the $a$-th bit of $\Phi(e, \rho)$ is the same for every leaf $\rho$ of $T_{k+1}^a$. We can use that to define the values of $g(e)$ for $n_k < e \leq n_{k+1}$ the following way: if the $a$-th bit of $\Phi(e, \rho)$ is $0$ for every leaf $\rho$ of $T_{k+1}^a$, then the $a$-th bit of $g$ is set to $1$, and vice-versa. This is here that we need to use the induction hypothesis $\exp(c) < G(n_k)$. It implies in particular that $\exp(c) < G(e)$ for $n_k < e \leq n_{k+1}$ (as $G$ is an order function). Also at most $c$ bits of $g(e)$ are set to something, which implies $g(e) \leq \exp(c)$ and thus $g(e) < G(e)$.

Let now $T_{k+1} = \bigcup_{a < c} T_{k+1}^a$. It is clear that $T_{k+1}$ is $G$-fat for $(n_0, \dots, n_{k+1})$. Let $d$ be the number of leaves of $T_{k+1}$. All we need to show now to continue the induction is that $\exp(d) < G(n_{k+1})$. To see this, let $b \leq c$ be such that $m_b \geq m_a$ for $a \leq c$. We now have by (4) of the claim that $\exp(w_{F}(m_b+\exp(n_{k+1} \cdot (k+1)))) < G(n_{k+1})$. Also $w_{F}(m_b+\exp(n_{k+1} \cdot (k+1)))$ is the number of nodes in the $F \rest {m_b + \exp(n_{k+1} \cdot (k+1))}$-full-branching tree above the empty string. And by the choice of $m_b$, for every $a \leq c$ we have that the tree $T_{k+1}^a$ is included in the $F \rest {m_b + (\exp(n_{k+1}\cdot (k+1))}$-full-branching tree above the empty string. It follows that for $d$ the number of leaves in $T_{k+1}$, we must have $d \leq w_{F}(m_b+\exp(n_{k+1}\cdot (k+1)))$ and thus that we must have $\exp(d) < G(n_{k+1})$.

The tree $T'$ is then defined to be $\bigcup_k T_k$. It is clear that by construction, the tree $T'$ is computable with no dead ends, infinitely often $G$-fat, and that for every path $X \in [T']$, we have $\Phi(X, n) \neq g(n)$ for every $e$.\\

Let us now give the proof of the claim. Figure~\ref{fig:2} illustrates a part of this proof. By hypothesis  $T$ is infinitely often $G$-fat. In particular, there exists $n_k > n_{k-1}$ such that above every $\sigma_a$, we have $m_a' \in \omega$ and an extension $\tau_a' \succ \s_a$ of length $m_a'$, such that $T$ is $F \rest {m_a' + \exp(n_k \cdot (k+1))}$-full-branching above $\tau_a'$ with 
$$\exp(w_{F}(m_a'+\exp(n_k \cdot (k+1))) < G(n_k)$$
Note that here, we truly mean $\exp(n_k \cdot (k+1))$ and not $\exp(n_k \cdot k)$. Given $\tau_a'$ of length $m_a'$, for each node of $T$ of length $m_a' + \exp(n_k \cdot (k+1))$ extending $\tau_a'$, we find an extension $\rho \in T$ of this node such that the values $\Phi(\rho, t)$ are defined for every $t$ with $n_{k-1} < t \leq n_k$. We define the tree ${T^{a}}'$ to be all these nodes $\rho$ and their prefixes. We now inductively apply \Cref{lemma:complicated} to the tree ${T^{a}}'$, so that for every $t$ with $n_{k-1} < t \leq n_k$, the $a$-th bit of $\Phi(t, \rho)$ is the same on every leaf $\rho$ of ${T^{a}}'$. Let us explain the first step. Given ${T^{a}}'$, we partition its leaves into these for which the $a$-th bit of $\Phi(n_{k-1}+1, \rho)$ is $0$, and these for which the $a$-th bit of $\Phi(n_{k-1}+1, \rho)$ is $1$. We then thin the tree ${T^{a}}'$ as described in \Cref{lemma:complicated}, so that the height of the full-branching part of ${T^{a}}'$ is halved, and the $a$-th bit of $\Phi(n_{k-1}+1, \rho)$ is the same for all the remaining leaves. We then inductively apply \Cref{lemma:complicated} on the successive resulting trees, to deal with the $a$-th bit of all the values $\Phi(t, \rho)$ for $n_{k-1} < t \leq n_k$. Let $T^a$ be the tree resulting of the successive applications from \Cref{lemma:complicated}. 

It is clear by design that (2) and (3) of the claim are satisfied. Let us verify  (1). Each each time we applied \Cref{lemma:complicated}, it halved the height of the full-branching part of ${T^{a}}'$. We applied \Cref{lemma:complicated} at most $n_k$ times. Also ${T^{a}}'$ is $F \rest {m_a' + \exp(n_k \cdot (k+1))}$-full-branching above $\tau_a'$. This means in particular that its full-branching part has height $\exp(n_k \cdot (k+1))$. It follows that the full-branching part of $T^a$ has height at least $\exp(n_k \cdot (k+1)) \exp(-n_k) = \exp(n_k \cdot k)$. Thus we have that $T^a$ is $F \rest {m_a + \exp(n_k \cdot k)}$-full branching above some node $\tau_a \succeq \tau_a'$ of length $m_a$. Thus also (1) is verified.

It remains to verify (4). Recall that $n_k$ and (for every $a$) the strings $\tau_a'$ of length $m_a'$ were picked such that
$$\exp(w_{F}(m_a'+(\exp(n_k \cdot (k+1)))) < G(n_k)$$   
In order to verify (4), we now want to show for every $a$ that:
$$\exp(w_{F}(m_a+(\exp(n_k \cdot k))) < G(n_k)$$
It suffices to show for every $a$ that $m_a + \exp(n_k \cdot k) \leq m_a' + \exp(n_k \cdot (k+1))$. Recall that $m_a$ is the length of the string $\tau_a$ extending $\tau_a'$, resulting of the successive applications of \Cref{lemma:complicated} to the full-branching part of ${T^{a}}'$. In particular we have $\tau_a \in {T^{a}}'$. Also the quantities $\exp(n_k \cdot (k+1))$ and $\exp(n_k \cdot k)$ are respectively the height of the full-branching part of ${T^{a}}'$ and the height of the full-branching part of $T^a \subseteq {T^{a}}'$. It easily follows that $m_a + \exp(n_k \cdot k) \leq m_a' + \exp(n_k \cdot (k+1))$.

\end{proof}

\begin{figure}
\begin{center}
\hspace*{0in}
\resizebox{\textwidth}{!}{\begin{tikzpicture}[every fit/.style={ellipse,draw,inner sep=-2pt},every text node part/.style={align=center}]
\tikzset{every tree node/.style={align=center}}
\tikzstyle{point}=[]
\tikzstyle{end}=[]
\tikzstyle{tiret}=[font=\bf, ultra thick]

\coordinate[] (start) {};
\coordinate[above=2cm of start] (end) {};
\coordinate[above left=4cm and 4cm of end] (t1) {};
\coordinate[right=1cm of t1] (t2) {};
\coordinate[above right=4cm and 4cm of end] (tn) {};

\coordinate[right=1cm of tn] (f1) {};
\coordinate[below=4cm of f1] (f2) {};

\coordinate[above=4cm of t1] (e1) {};
\coordinate[above=4cm of t2] (e2) {};
\coordinate[above=4cm of tn] (en) {};

\coordinate[above=.5cm of tn] (m1) {};
\coordinate[above right=2cm and .5cm of m1] (m3) {};

\coordinate[right=2cm of m1] (s1) {};
\node[right=7pt of s1] {$\Phi(\rho, n_{k-1}+1)$};
\coordinate[above=1cm of s1] (s2) {};
\node[right=7pt of s2] {$\Phi(\rho, n_{k-1}+2)$};
\coordinate[above=1cm of s2] (s3) {};
\node[right=7pt of s3] {\ \ \ $\dots$};
\coordinate[above=1cm of s3] (s4) {};
\node[right=7pt of s4] {$\Phi(\rho, n_{k})$};

\coordinate[above right=4.25cm and 2cm of s1] (sns) {};

\coordinate[left= 6cm of start] (left1) {};
\coordinate[above= 10cm of left1] (left2) {};

\coordinate[above =2cm of left1] (map);
\node[left=7pt of map] {$m_a'$};

\coordinate[above=6cm of left1] (map2);
\node[left=7pt of map2] {$m_a'+$\\$2^{n_k.(k+1)}$};

\coordinate[above=3.5cm of left1] (ma);
\node[left=7pt of ma] {$m_a$};

\coordinate[above =5cm of left1] (ma2);
\node[left=7pt of ma2] {$m_a+2^{n_k.k}$};

\node[above left=1cm and 7pt of start] {$\tau_a'$};

\coordinate[above right=1.5cm and .5 of end] (ends) {};
\coordinate[above left=1.5cm and 1.5cm of ends] (t1s) {};
\coordinate[right=1cm of t1s] (t2s) {};
\coordinate[above right=1.5cm and 1.5cm of ends] (tns) {};

\coordinate[above=5cm of t1s] (e1s) {};
\coordinate[above=5cm of t2s] (e2s) {};
\coordinate[above=5cm of tns] (ens) {};

\node[above left=.5cm and 1pt of end] {$\tau_a$};

\draw [-,decorate,decoration={snake,amplitude=.6mm,segment length=6mm,post length=1mm}] (start) -- (end);
\draw [-] (end) -- (t1);
\draw [-] (end) -- (tn);
\draw [-] (t1) -- (tn);

\draw [-,decorate,decoration={snake,amplitude=.6mm,segment length=6mm,post length=1mm}] (end) -- (ends);
\draw [-] (ends) -- (t1s);
\draw [-] (ends) -- (tns);
\draw [-] (t1s) -- (tns);

\draw [-] (t1) -- (e1);
\draw [-] (t2) -- (e2);
\draw [-] (tn) -- (en);

\draw [-,decorate,decoration={snake,amplitude=.3mm,segment length=6mm,post length=1mm}] (t1s) -- (e1s);
\draw [-,decorate,decoration={snake,amplitude=.3mm,segment length=6mm,post length=1mm}] (t2s) -- (e2s);
\draw [-,decorate,decoration={snake,amplitude=.3mm,segment length=6mm,post length=1mm}] (tns) -- (ens);


\draw[-, dotted, thick, shorten <=-5pt,shorten >=-5pt] (e1s) to[bend left] node[above] (boo) {Leaves $\rho$ of $T_a$} (ens);

\draw[-, dotted, thick, shorten <=5pt] (tn) to[bend right] node[above] {} (en);
\draw[->,shorten <=5pt,shorten >=5pt] (m3) to node[above]{$\Phi(\rho, t)$} (s3);

\draw[->,shorten <=5pt,shorten >=5pt] (boo) to node[above]{For $n_{k-1} < t \leq n_k$\\same $a$-th bit of $\Phi(\rho, t)$ for all $\rho$} (sns);

\draw[-, dashed] (left1) to (left2);
\draw[-, dashed,shorten <=5pt,shorten >=5pt] (map) to (end);
\draw[-, dashed,shorten <=5pt,shorten >=5pt] (map2) to (t1);
\draw[-, dashed,shorten <=5pt,shorten >=5pt] (ma) to (ends);
\draw[-, dashed,shorten <=5pt,shorten >=5pt] (ma2) to (t1s);

\draw[<->, dashed] (f1) to node[right]{full-branching\\ part} (f2);

\end{tikzpicture}}
\end{center}
\caption{Construction of $T^{a}$ in the claim of \Cref{lemma:all}. $T_a$ is the result of thinning the full-branching part above $\tau_a'$ into a full-branching part above a extension $\tau_a \succeq \tau_a'$}.
\label{fig:2}
\end{figure}
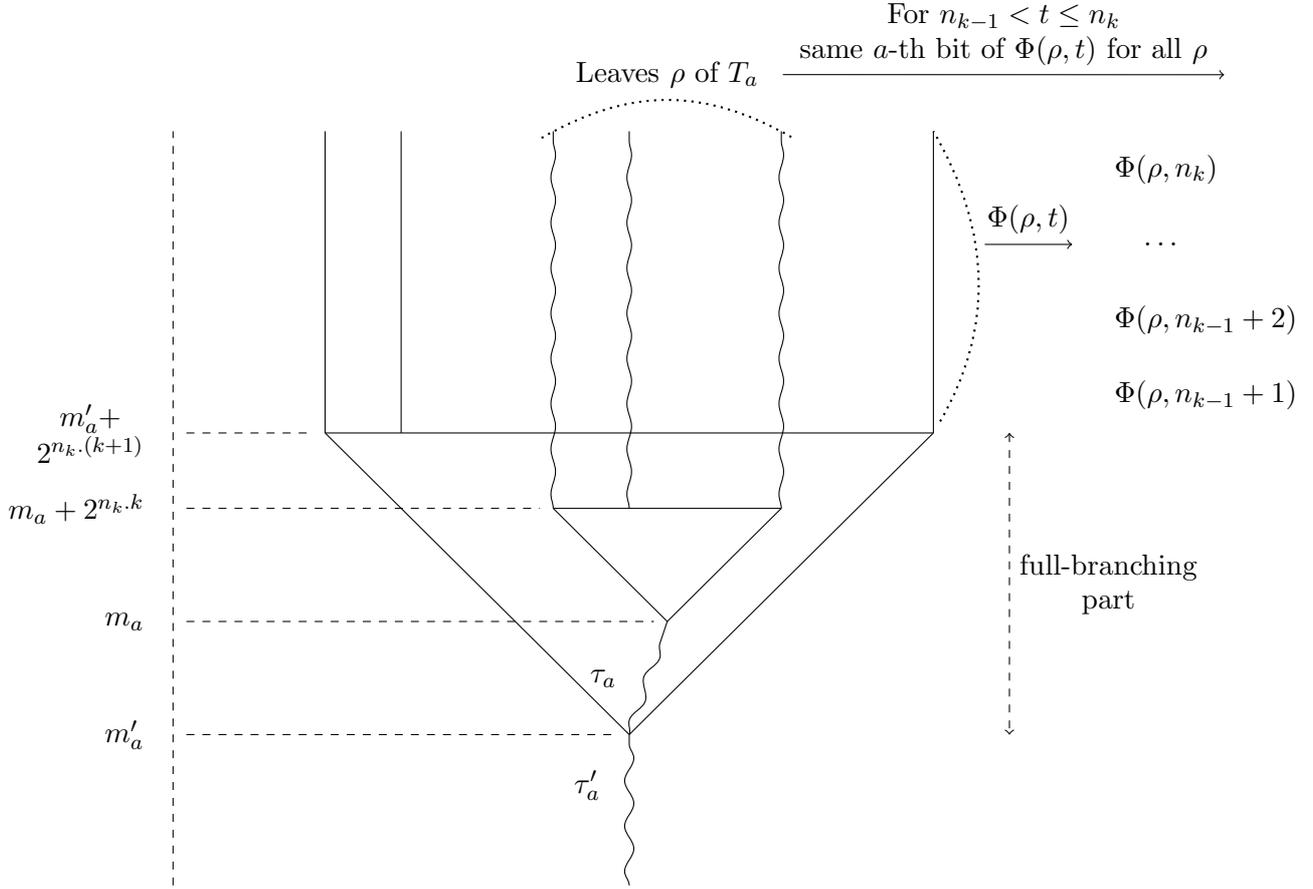
\section{Some open questions}
\Cref{thm:mainDelta} implies that there are no $\Gamma$-values strictly between $0$ and $1/2$. However, if  $\Gamma(X) < 1/2$, the lemma does not provide a single set $Y \leq_T X$ such that $\Gamma(Y) = 0$. 
\begin{question} \label[question]{qu:gamma0}
Let $X$ be a set such that  $\Gamma(X) = 0$. Is there   a set $Y \leq_T X$ such that $\gamma(Y) = 0$? 
Equivalently, is $\+ D(0)  \equiv_W \+ D(1/4)$?
\end{question}
This question is actually connected to other questions regarding the hierarchy of mass problem $\mathrm{IOE}(h)$ in the Muchnik degrees. We showed in \Cref{th-ioe} that this hierarchy is proper, but given $f$, the function  $g > f$ we provide such that $\mathrm{IOE}(f) <_W \mathrm{IOE}(g)$ grows  rather fast compared to $f$. We do not know for instance if given any $f$ we have $\mathrm{IOE}(f) <_W \mathrm{IOE}(\lambda n .  f(n^2))$. So we ask here the following question:
\begin{question}
Does there existe a computable function $f$ with

  $\forall a \in \omega\ \forall^\infty n\ f(n) > 2^{(a^n)}$   such that
  $\mathrm{IOE}(2^{(2^n)}) \equiv_W \mathrm{IOE}(f)$ ? 
\end{question}
A positive answer to this question would also provide a positive answer to \Cref{qu:gamma0}.  For, by Remark \ref{rem:D0}, $ \mathrm{IOE}(f) \ge_S \+ D(0)$, and we have $\+ D(0) \ge_S \+ D(1/4) \equiv_W \mathrm{IOE}(2^{(2^n)})$. 

There is  also a question regarding the sets $X$ such that $\Gamma(X) = 1/2$.
\begin{question}
Let $X$ be a set with  $\Gamma(X) = 1/2$. Let $Y \leq_T X$. Is there  a computable set $A$ such that $\uul \rho (Y \lra A) = 1/2$ ?
\end{question}
%
Again, the proof of \Cref{thm:mainDelta} does not help answering this question. All  we  have is an affirmative answer to the question for all the known examples of sets $X$ with a $\Gamma$-value of $1/2$.

Finally we ask for an analog of Theorem~\ref{th-ioe} for cardinal characteristics. 
\begin{question} Given an order function $F$, is there a faster growing  order function $G$ such that $\frd(\neq^*, G) < \frd(\neq^* , F) $ is consistent with ZFC? \end{question}

\def\cprime{$'$} \def\cprime{$'$}

%
%
%
%
%

\end{document}